\documentclass{amsart}
\usepackage{amsmath}
\usepackage{amssymb}
\usepackage{amsthm}
\usepackage[dvips]{graphicx}

\def\NZQ{\mathbb}
\def\QQ{{\NZQ Q}}

\def\frk{\mathfrak}               

\def\frm{{\frk m}}

\def\bbN{{\mathbf N}}
\def\bbZ{{\mathbf Z}}
\def\opn#1#2{\def#1{\operatorname{#2}}} 
\opn\chara{char} \opn\length{\ell} \opn\pd{pd} \opn\rk{rk}
\opn\projdim{proj\,dim} \opn\injdim{inj\,dim} \opn\rank{rank}
\opn\depth{depth} \opn\codepth{codepth} \opn\grade{grade}
\opn\height{height} \opn\embdim{emb\,dim} \opn\codim{codim}

\opn\Tr{Tr} \opn\bigrank{big\,rank}
\opn\superheight{superheight}\opn\lcm{lcm}
\opn\trdeg{tr\,deg}%
\opn\reg{reg} \opn\lreg{lreg} \opn\skel{skel} \opn\Gr{Gr}
\opn\dim{dim} \opn\indeg{indeg} \opn\Ass{Ass} \opn\Min{Min}
\opn\div{div} \opn\Div{Div} \opn\cl{cl} \opn\Cl{Cl}
\opn\Spec{Spec} \opn\Supp{Supp} \opn\supp{supp} \opn\Sing{Sing}
\opn\Ass{Ass}
\opn\Ann{Ann} \opn\Rad{Rad} \opn\Soc{Soc}
\opn\Sym{Sym} \opn\Ker{Ker} \opn\Coker{Coker} \opn\Im{Im}
\opn\Hom{Hom} \opn\Tor{Tor} \opn\Ext{Ext} \opn\End{End}
\opn\Aut{Aut} \opn\id{id} \opn\ini{in} \opn\tr{tr}

\opn\nat{nat}\opn\it{it}
\opn\pff{proof}
\opn\Pf{proof} \opn\GL{GL} \opn\SL{SL} \opn\mod{mod} \opn\ord{ord}
\opn\diam{diam}
\opn\aff{aff} \opn\con{conv} \opn\relint{relint} \opn\st{st}
\opn\lk{lk} \opn\cn{cn} \opn\core{core} \opn\vol{vol}
\opn\link{link} \opn\star{star} \opn\skel{skel}
\opn\gr{gr}

\def\pot#1#2{#1[\kern-0.28ex[#2]\kern-0.28ex]}

\opn\dirlim{\underrightarrow{\lim}}
\opn\inivlim{\underleftarrow{\lim}}
\let\to=\rightarrow

\def\Implies{\ifmmode\Longrightarrow \else
     \unskip${}\Longrightarrow{}$\ignorespaces\fi}
\def\implies{\ifmmode\Rightarrow \else
     \unskip${}\Rightarrow{}$\ignorespaces\fi}
\def\iff{\ifmmode\Longleftrightarrow \else
     \unskip${}\Longleftrightarrow{}$\ignorespaces\fi}

\let\:=\colon
\theoremstyle{plain}
\newtheorem{thm}{Theorem}[section]
\newtheorem{lemma}[thm]{Lemma}
\newtheorem{prop}[thm]{Proposition}
\newtheorem{cor}[thm]{Corollary}

\newtheorem{quest}[thm]{Question}

\newtheorem*{thm-q}{Theorem}
\newtheorem*{cor-q}{Corollary}
\newtheorem*{quest-q}{Question}
\newtheorem*{quests-q}{Questions}
\theoremstyle{definition}
\newtheorem{defn}[thm]{Definition}
\newtheorem{exam}[thm]{Example}

\newtheorem*{acknowledgement}{Ackowledgement}
\theoremstyle{remark}
\newtheorem{remark}[thm]{Remark}

\let\epsilon\varepsilon
\let\phi=\varphi
\let\kappa=\varkappa
\def\qed{\ifhmode\textqed\fi
   \ifmmode\ifinner\quad\qedsymbol\else\dispqed\fi\fi}
\def\textqed{\unskip\nobreak\penalty50
    \hskip2em\hbox{}\nobreak\hfil\qedsymbol
    \parfillskip=0pt \finalhyphendemerits=0}
\def\dispqed{\rlap{\qquad\qedsymbol}}

\opn\Gin{Gin}

\def\FF{{\mathcal F}}

\opn\inii{in} \opn\inim{inm} \opn\rate{rate}
\numberwithin{equation}{section}

\begin{document}
\title{Cohen--Macaulaynees for symbolic power ideals of edge ideals}
\author[Giancarlo Rinaldo]{Giancarlo Rinaldo}
\address[Giancarlo Rinaldo]{Dipartimento di Matematica, 
Universita' di Messina, Salita Sperone, 31. S. Agata, 
Messina 98166, Italy}
\email{rinaldo@dipmat.unime.it}
\author[Naoki Terai]{Naoki Terai}
\address[Naoki Terai]{Department of Mathematics, Faculty of Culture 
and Education, Saga University, Saga 840--8502, Japan}
\email{terai@cc.saga-u.ac.jp}
\author[Ken-ichi Yoshida]{Ken-ichi Yoshida}
\address[Ken-ichi Yoshida]{Graduate School of Mathematics, Nagoya University, 
         Nagoya 464--8602, Japan}
\email{yoshida@math.nagoya-u.ac.jp}
\subjclass[2000]{Primary 13F55, Secondary 13H10}
\date{\today}
\keywords{edge ideal, complete intersection, 
Cohen--Macaulay, FLC, symbolic powers, polarization, simplicial complex}

\begin{abstract}
Let $S = K[x_1, \ldots, x_n]$ be a polynomial ring over a field $K$.
Let $I(G)\subseteq S$ denote the edge ideal of a graph $G$.  
We show that the $\ell$th symbolic power 
$I(G)^{(\ell)}$ is a Cohen--Macaulay ideal 
(i.e., $S/I(G)^{(\ell)}$ is Cohen--Macaulay) 
for some integer $\ell \ge 3$ if and only if 
$G$ is a disjoint union of finitely many complete graphs.  
When this is the case, all the symbolic powers $I(G)^{(\ell)}$ 
are Cohen--Macaulay ideals. 
Similarly, we characterize graphs $G$ for which 
$S/I(G)^{(\ell)}$ has (FLC). 
\par 
As an application, we show that an edge ideal $I(G)$ 
is complete intersection provided that $S/I(G)^{\ell}$ is 
Cohen--Macaulay for some integer $\ell \ge 3$. 
This strengthens the main theorem in \cite{CRTY}.  
\end{abstract}

\maketitle
\setcounter{section}{-1}
\section{Introduction}

\par 
In this paper, we restrict our attention to the edge ideals of graphs. 
For a graph $G=(V(G),E(G))$, the edge ideal, denoted by $I(G)$, is defined by 
\[
 I(G) = (x_ix_j \; :\; \{x_i,x_j\} \in E(G))S,  
\]
where $S = K[v : v \in V(G)] = K[x_1, \ldots, x_n]$ is 
a polynomial ring over a field $K$. 
Then $E(G)$ is a squarefree monomial ideal which is generated 
by degree $2$ elements, and thus it can be regarded as a Stanley--Reisner ideal 
and it is a radical ideal.   
Then the following theorem is well-known. 

\begin{thm-q}[See \cite{AV,CN, Wa}] \label{CowsikNori}
Let $S$ be a regular local ring $($resp., a polynomial ring over a field $K$$)$,   
and let $I$ be a radical ideal $($resp., a homogeneous radical ideal$)$ of $S$. 
Then $I$ is complete intersection if and only if $S/I^{\ell}$ is 
Cohen--Macaulay for infinitely many $\ell \ge 1$. 
 \par 
In particular, for any edge ideal $I(G)$ of a graph $G$, 
$I(G)$ is a complete intersection ideal if and only if  
$S/I(G)^{\ell}$ is 
Cohen-Macaulay for infinitely many $\ell \ge 1$. 
\end{thm-q}

\par  
In what follows, let $G=(V(G),E(G))$ be a graph, and 
$I(G) \subseteq S=K[v \,:\, v\in V(G)]$ the edge ideal of $G$. 

\par 
Recently, in \cite{TY}, the last two authors gave a generalization of 
the theorem using a classification theorem for locally complete intersection 
Stanley--Reinser ideals; see \cite[Theorem 1.15]{TY}. 
Note that the following theorem is also true for Stanley--Reisner ideals. 

\begin{thm-q}[\textrm{See \cite[Theorem 2.1]{TY}}] \label{BbmTY}
If $S/I(G)^{\ell}$ is Buchsbaum for infinitely many 
$\ell \ge 1$, then $I(G)$ is complete intersection.  
\end{thm-q}

\par
Moreover, the authors \cite{CRTY} gave a refinement of 
the above theorem jointly with M. Crupi.  

\begin{thm-q}[\textrm{See \cite[Theorem 2.1]{CRTY}}] \label{CRTY}
$I(G)$ is complete intersection if and only if 
$S/I(G)^{\ell}$ is Cohen--Macaulay for some $\ell \ge \height I$. 
\end{thm-q}

\par
The main purpose of this paper is to give another variation of the theorem
in this context.  
Namely, we consider the following questions:

\begin{quests-q} \label{Q-CM}
Let $\ell \ge 1$ be an integer. 
Let $I(G)^{(\ell)}$ denote the $\ell$th symbolic power ideal of $I(G)$. 
Then$:$ 
\begin{enumerate}
\item When is $S/I(G)^{(\ell)}$ Cohen--Macaulay? 
\item Is $I(G)$ complete intersection if $S/I(G)^{\ell}$ Cohen--Macaulay 
for a fixed $\ell \ge 1$? 
\end{enumerate}
\end{quests-q}

\par 
The answers to these questions will give a generalization 
of the original theorem described as above. For instance,  
for each fixed $\ell \ge 1$, 
the Cohen--Macaulayness of $S/I(G)^{\ell}$ implies 
that of $S/I(G)^{(\ell)}$.  
Note that the converse is \textit{not} true in general.

\par 
We first consider the above question. 
Let $G$ be a graph on the vertex set $V=[n]$ such that 
$\dim S/I(G)=1$. 
Such a graph $G$ is isomorphic to the complete graph $K_n$. 
Then $S/I(G)^{(\ell)}$ is Cohen--Macaulay for every integer 
$\ell \ge 1$ because the symbolic power ideal 
has no embedded primes.  

\par 
The following theorem 
characterizes graphs $G$  
for which all symbolic powers $S/I(G)^{(\ell)}$ are Cohen--Macaulay
(or for $\ell \ge 3$). 

\begin{thm-q}[See Theorem \ref{Main-CM}] \label{Main1}  
The following conditions are equivalent$:$ 
\begin{enumerate}
 \item $S/I(G)^{(\ell)}$ is Cohen--Macaulay for every integer $\ell \ge 1$. 
 \item $S/I(G)^{(\ell)}$ is Cohen--Macaulay for some $\ell \ge 3$. 
 \item $S/I(G)^{(\ell)}$ satisfies Serre's condition $(S_2)$ for some $\ell \ge 3$. 
 \item $G$ is a disjoint union of finitely many complete graphs. 
\end{enumerate}  
\end{thm-q}

\par \vspace{2mm}
As an application of the theorem, we can obtain some result 
for Cohen--Macaulayness of ordinary powers, 
which gives an improvement of the main theorem 
in \cite{CRTY}. 

\begin{cor-q}[See Theorem \ref{Power-cor}] \label{Power-cor-q}
If $S/I(G)^{\ell}$ is Cohen--Macaulay 
for some $\ell \ge 3$, 
then $I(G)$ is complete intersection. 
\end{cor-q}

\par \vspace{2mm}
Next, we consider the following question. 
We need to assume that $I(G)$ is unmixed. 
Then if $\dim S/I(G) \le 2$,    
for every integer $\ell \ge 1$, 
$S/I(G)^{(\ell)}$ is unmixed, and thus it has (FLC). 

\begin{quest-q} \label{Q-FLC} 
Let $\ell \ge 1$ be an integer. 
When does $S/I(G)^{(\ell)}$ have $($FLC$)$?
\end{quest-q} 

\par 
Let $\Delta=\Delta_{n_1,\ldots,n_r}$ denote the simplicial complex 
whose Stanley--Reisner ideal is equal to 
the edge ideal of the disjoint union of 
complete graphs $K_{n_1},\ldots,K_{n_r}$. 
That is, 
\[
 I_{\Delta_{n_1,\ldots,n_r}} = 
I(K_{n_1} \textstyle{\coprod} \cdots \textstyle{\coprod} K_{n_r}).
\]  
\par 
Then the following theorem gives an answer to the above question 
for $\ell \ge 3$:

\begin{thm-q}[See Theorem \ref{Main-FLC}] \label{Main2} 
Let $\Delta(G)$ be the simplicial complex on $V(G)$ which satisfies 
$I_{\Delta(G)} = I(G)$.   
Suppose that $\Delta(G)$ is pure and $d=\dim S/I(G) \ge 3$. 
Let $p$ denote the number of connected components of $\Delta(G)$. 
Then the following conditions are equivalent$:$ 
\begin{enumerate}
 \item $S/I(G)^{(\ell)}$ has $($FLC$)$ for every integer $\ell \ge 1$. 
 \item $S/I(G)^{(\ell)}$ has $($FLC$)$ for some $\ell \ge 3$. 
 \item There exist  
$(n_{i1},\ldots,n_{id}) \in \bbN^d$ for every $i=1,\ldots,p$
such that 
$\Delta$ can be written as 
\[
\Delta = \Delta_{n_{11},\ldots,n_{1d}} \;\textstyle{\coprod}\; 
\Delta_{n_{21},\ldots,n_{2d}} \;\textstyle{\coprod} 
\; \ldots \;
\textstyle{\coprod} \;
\Delta_{n_{p1},\ldots,n_{pd}}.
\] 
\end{enumerate}  
\end{thm-q} 

\par 
For $\ell=2$, the problem is more complicated. 
For instance, if $G$ is a pentagon, then $I(G)$ and $I(G)^2$ are Cohen--Macaulay 
although $I(G)^{(\ell)}$ (and hence $I(G)^{\ell}$) is \textit{not} for any $\ell \ge 3$.   

\par 
After finishing this work the authors have known that N. C. Minh obtained similar results independently. 

\bigskip
\section{Preliminaries}
In this section we recall some definitions and properties that we
will use later. 

\subsection{Edge ideals}
Let $G$ be a graph, which means a simple finite graph without loops and 
multiple edges.  
Let $V(G)$ (resp., $E(G)$) denote the set of vertices (resp., edges) of $G$. 
Put $V(G) =\{x_1,\ldots, x_n\}$. 
Then the \textit{edge ideal} of $G$, denoted by $I(G)$, 
is a squarefree monomial ideal 
of $S=K[x_1,\ldots,x_n]$ defined by 
\[
 I(G) = (x_ix_j \,:\, \{x_i,x_j\} \in E(G)).
\]
\par 
A disjoint union of two graphs $G_1$ and $G_2$, denoted by 
$G_1 \coprod G_2$, is the graph $G$ which satisfies $V(G) = V(G_1) \cup V(G_2)$ and $E(G)=E(G_1) \cup E(G_2)$.   
For a nonempty subset $W \subseteq V(G)$, $H=G\,|_{W}$ denotes the graph which satisfies 
$V(H)=W$ and $E(H) = \{\{x,y\} \in E(G) \,:\, x,y \in W \}$.        

\par \vspace{2mm}
\subsection{Stanley--Reisner ideals}
Let $V=\{x_1, \ldots, x_n\}$. 
A nonempty subset $\Delta$ of the power set $2^V$ 
is called a \textit{simplicial complex} on $V$ if 
$\{v\} \in \Delta$ for all $v \in V$, 
and $F \in \Delta$, $H \subseteq F$ imply $H \in \Delta$. 
An element $F \in \Delta$ is called a \textit{face} of $\Delta$.  
The dimension of $\Delta$ is defined by 
$\dim \Delta = \max\{\sharp(F)-1 \,:\, \text{$F$ is a face of $\Delta$}\}$. 
A maximal face of $\Delta$ is called a \textit{facet} of $\Delta$. 
$\mathcal{F}(\Delta)$ denote the set of all facets of $\Delta$. 
The \textit{Stanley--Reisner ideal} of $\Delta$, denoted by $I_{\Delta}$,   
is the squarefree monomial ideal generated by 
\[
 \{x_{i_1} x_{i_2} \cdots x_{i_p} \,:\, 1 \le i_1 < \cdots < i_p \le n,\; 
\{x_{i_1},\ldots,x_{i_p}\} \notin \Delta \},  
\]
and $K[\Delta]= K[x_1,\ldots,x_n]/I_{\Delta}$ is called 
the \textit{Stanley--Reisner ring} of $\Delta$. 

\par 
For an arbitrary graph $G$, the simplicial complex $\Delta(G)$ 
with $I(G) = I_{\Delta(G)}$ is called 
the \textit{complementary simplicial complex} of $G$.

\par
Put $d=\dim \Delta+1$.  
A simplicial complex $\Delta$ is called \textit{pure}  
if all the facets of $\Delta$ have the same cardinality $d$. 
A pure simplicial complex $\Delta$ is \textit{connected in codimension $1$}
(or \textit{strongly connected}) if for every two facets $F$ and $H$ of
$\Delta$, there is a sequence of facets 
$F = F_0, F_1, \ldots, F_m =H$ 
such that $\sharp (F_i \cap F_{i+1}) = d-1$ for each $i =0,\ldots, m-1$.
For every face $F \in \Delta$, the $\textit{star}$ and the
\textit{link} of $F$ are defined by:
\begin{eqnarray*}
\star_{\Delta} F &=& \{H \in \Delta \;:\,H \cup F \in \Delta\}, \\
\link_{\Delta} F  &=& \{H \in \Delta\;:\,H \cup F \in \Delta,\, H \cap F = \emptyset\}.
\end{eqnarray*}
Note that these are also simplicial complexes. 

\par \vspace{2mm}
\subsection{Serre's condition}
Let $S=K[x_1,\ldots,x_n]$ and $\frm = (x_1,\ldots,x_n)S$. 
Let $I$ be a homogeneous ideal of $S$. 
For a positive integer $k$,   
$S/I$ satisfies \textit{Serre's condition} $(S_k)$ 
if $\depth (S/I)_P \ge \min\{\dim (S/I)_P,\, k \}$ for every $P \in \Spec S/I$.
\par 
The ring $S/I$ is called \textit{Cohen--Macaulay} 
if $\depth S/I = \dim S/I$.
This is an equivalent condition that $S/I$ satisfies Serre's condition $(S_d)$, 
where $d= \dim S/I$.    
Moreover, the ring $S/I$ is called (\textit{FLC}) if 
$H_{\frm}^i(S/I)$ has finite length for every $i \ne \dim S/I$.    
The ring $S/I$ is called \textit{Buchsbaum} if 
the natural map $\Ext_S^i(S/\frm,S/I) \to H_{\frm}^i(S/I)$ is surjective 
for every $i \ne \dim S/I$.   
Note that any Cohen-Macaulay ring is Buchsbaum, and any Buchsbaum ring has (FLC).

\par 
A simplicial complex $\Delta$ is called 
\textit{Cohen--Macaulay} (resp., \textit{Buchsbaum}, \textit{FLC})
if so is $K[\Delta]$. 
Note that $\Delta$ is Buchsbaum if and only if it satisfies (FLC). 
Moreover, 
if $\Delta$ is (FLC), then $\Delta$ is pure and 
$\link_{\Delta}(F)$ is Cohen-Macaulay for every nonempty face $F \in \Delta$. 
\par 
We notice that $\Delta$ is pure and     
connected in codimension $1$ 
if $K[\Delta]$ satisfies $(S_2)$ and $\dim \Delta \ge 1$. 

\par \vspace{2mm}
\subsection{Takayama's formula}
Let $I$ be an arbitrary monomial ideal in $S=K[x_1,\ldots,x_n]$. 
Then the $i$th local cohomology module $H_{\frm}^i(S/I)$ can be regarded as a $\bbZ^n$-module over $S/I$. 
For every ${\bf a} =(a_1,\ldots,a_n) \in \bbZ^n$, we set 
$G_a = \{i \;:\; a_i < 0 \}$ and define 
\[
\Delta_{\bf a}(I) = 
\{F \subseteq [n] \;:\; \text{$F$ satisfies $(C1)$ and $(C2)$}\},  
\]
where 
\begin{enumerate}
 \item[(C1)] $F \cap G_{\bf a} = \emptyset$. 
 \item[(C2)] for every minimal generator $u=x_1^{c_1}\cdots x_n^{c_n}$ of $I$ 
there exists an index $i \notin F \cup G_{\bf a}$ with $c_i > a_i$.  
\end{enumerate}
Moreover, we define 
\[
\Delta(I) = \{F \subseteq [n] \;:\; 
\textstyle{\prod_{i \in F}} x_i  \notin \sqrt{I}\}. 
\]
Then $\Delta(I)$ is a simplicial complex and 
$\Delta_{\bf a}(I)$ is a subcomplex of $\Delta(I)$ with 
$\dim \Delta_{\bf a}(I) = \dim \Delta(I)-\sharp(G_a)$ 
provided that $\Delta(I)$ is pure and $\Delta_{\bf a}(I) \ne \emptyset$ 
similarly as in \cite[Lemma 1.3]{MiT}. 
\par \vspace{2mm}
Now let us recall Takayama's formula, which is a generalization of 
well-known Hochster's formula. 

\begin{lemma}[\textbf{Takayama's formula}; 
\textrm{see e.g. \cite[Theorem 1.1]{MiT}}] \label{Takayama}
Let $I$ be an arbitrary monomial ideal in $S=K[x_1,\ldots,x_n]$. 
For every ${\bf a} \in \bbZ^n$, we have 
\[
\dim_K H_{\frm}^i(S/I)_{\bf a} = 
\left\{
\begin{array}{ll}
\dim_K \widetilde{H}_{i-\sharp(G_{\bf a})-1} (\Delta_{\bf a}(I)), 
& \text{if $G_{\bf a} \in \Delta(I)$}, \\ 
0, & \text{else}. 
\end{array}
\right. 
\]
\end{lemma}

\par 
Using this lemma, we obtain the following criterion for Cohen-Macaulayness 
of $S/I$; see also \cite{MiT} in the case where 
$I=I_{\Delta}^{(\ell)}$ and $\dim S/I_{\Delta} = 1$. 

\begin{prop}  \label{criterion}
The following conditions are equivalent$:$
\begin{enumerate}
 \item $S/I$ is Cohen-Macaulay. 
 \item $S/I$ has $($FLC$)$, and for any ${\bf a} \in \bbN^n$, 
 we have that $\widetilde{H}_i(\Delta_{\bf a}(I))=0$ for all 
 $i < \dim \Delta_{\bf a}(I)$. 
\end{enumerate}
\end{prop}

\begin{proof}
$(1)\Longrightarrow (2):$ 
Since $S/I$ is Cohen-Macaulay, it has (FLC). 
For any ${\bf a} \in \bbN^n$, we have 
\[
 \widetilde{H}_i(\Delta_{\bf a}(I)) \cong H_{\frm}^{i+1}(S/I)_{\bf a} =0
\]
for all $i < \dim \Delta_{\bf a}(I)= \dim \Delta(I) = \dim S/I-1$
by Lemma \ref{Takayama} since $S/I$ is Cohen-Macaulay. 
\par \vspace{2mm} \par \noindent 
$(2)\Longrightarrow (1):$ 
Since $S/I$ has (FLC), $S/\sqrt{I}$ has also (FLC) and $\Delta(I)$ is pure; 
see \cite{HTT}. 
\par
Suppose that $S/I$ is \textit{not} Cohen-Macaulay. 
For any ${\bf a} \in \bbN^n$ we have 
\[
H_{\frm}^i(S/I)_{\bf a} \cong 
\widetilde{H}_{i-1}(\Delta_{\bf a}(I))=0
\]
for all $i \le \dim \Delta_{\bf a}(I) = \dim \Delta(I)$. 
So there exist a vector ${\bf a} \in \bbZ^n \setminus \bbN^n$ and 
an index $i \le \dim \Delta(I)$ such that 
\[
\widetilde{H}_{i-\sharp(G_{\bf a})-1}(\Delta_{\bf a}(I)) 
\cong H_{\frm}^i(S/I)_{\bf a} \ne 0. 
\]
Set ${\bf a} = (a_1,\ldots,a_n)$ and $a_j < 0$. 
Take any integer $k > 0$ and set ${\bf b} = {\bf a} - k {\bf e}_j$, 
where ${\bf e}_j$ is the $j$th unit vector.  
Then we have $\Delta_{\bf a}(I)=\Delta_{\bf b}(I)$ 
because $G_{\bf a} = G_{\bf b}$. 
In particular, $H_{\frm}^i(S/I)_{\bf b} \ne 0$. 
But this contradicts the assumption that $S/I$ has (FLC).  
\end{proof}

\par \vspace{2mm}
\subsection{Symbolic power ideals}
Let $I$ be a radical ideal of $S$. 
Let $\Min_S(S/I) = \{P_1,\ldots, P_r\}$ be the set of the 
minimal prime ideals of $I$, and 
put $W = S \setminus \bigcup_{i=1}^r P_i$. 
Given an integer $\ell \ge 1$, 
the \textit{$\ell$th symbolic power} of $I$ 
is defined to be the ideal
\[
 I^{(\ell)}= I^{\ell}S_W \cap S = \bigcap_{i=1}^r P_i^{\ell}S_{P_i} \cap S.
\]
In particular, if $I$ is a squarefree monomial ideal of $S$, 
then one has 
\[
 I^{(\ell)} = P_1^{\ell} \cap \cdots \cap P_r^{\ell}. 
\]
\par 
Let $\Delta$ be an arbitrary simplicial complex on $V=[n]$, and let 
$I_{\Delta} \subseteq S=K[x_1,\ldots,x_n]$ denote 
the Stanley-Reisner ideal of $\Delta$. 
For any integer $\ell \ge 1$ and ${\bf a} \in \bbN^{n}$, we set 
\[
 \Delta^{(\ell)}_{\bf a} = \langle F \in \mathcal{F}(\Delta) \;:\; 
\sum_{t \in V \setminus F} a_t \le \ell-1 \rangle. 
\] 

\par
We use the following remark and Proposition \ref{criterion} in 
the proof of Theorem \ref{SPcomplete}. 

\begin{remark} \label{MIT=TAK}
Under the notation above, for any ${\bf a} \in \bbN^n$, we have  
\begin{enumerate}
\item $\Delta_{\bf a}^{(\ell)} = \Delta_{\bf a}(I_{\Delta}^{(\ell)})$; 
see \cite[Section 1]{MiT}.
\item If $\Delta$ is pure and $\Delta_{\bf a}^{(\ell)} \ne \emptyset$ then 
$\dim \Delta_{\bf a}^{(\ell)} = \dim \Delta$. 
\end{enumerate}
\end{remark}

\par \vspace{2mm}
\subsection{Polarizations}
Now let $u = x_1^{a_1}x_2^{a_2} \cdots x_n^{a_n}$ be a monomial in $S$. 
Then we can associate to 
it a squarefree monomial $u^{\mathrm{pol}}$ as follows:
In the polarization process, each power of a variable $x_i^{a_i}$ is
replaced by a product of $a_i$ new variables $x_{i}^{(j)}$, 
$i \in \{1,\ldots, n\}$, $j \in \{0, 1, \ldots, a_i-1\}$:
\[
u^{\mathrm{pol}} = x_{1}^{(0)}x_{1}^{(1)}\cdots x_{1}^{(a_1-1)}
x_{2}^{(0)} x_{2}^{(1)} \cdots x_{2}^{(a_2-1)}
\cdots 
x_{n}^{(0)} x_{n}^{(1)} \cdots x_{n}^{(a_n-1)},
\] 
where all $x_i^{(j)}$ are distinct variables and 
$x_i^{(0)}=x_i$ for each $i$. 
We call $u^{\mathrm{pol}}$ the \textit{polarization} of 
$u$ (see \cite{SV}). 
Let $I =(u_1,\ldots,u_s)$ be a monomial ideal of $S$, 
where $\{u_1, \ldots,u_s\}$ is the minimal set 
of monomial generators of $I$. 
If $S^{\mathrm{pol}}$ is a polynomial ring over $K$ 
containing all monomials 
$u_1^{\mathrm{pol}},\ldots, u_s^{\mathrm{pol}}$, 
then we can consider the ideal 
$I^{\mathrm{pol}} = 
(u_1^{\mathrm{pol}}, \ldots, u_s^{\mathrm{pol}})$ 
of $S^{\mathrm{pol}}$. 
It is known that, for monomial ideals $I$ and $J$, one has (\cite{SV})
\begin{equation}\label{inter}
    (I \cap J)^{\mathrm{pol}} = I^{\mathrm{pol}} \cap J^{\mathrm{pol}}.
\end{equation}
\par 
It is well-known that if $S/I$ is Cohen--Macaulay then 
so is $S^{\mathrm{pol}}/I^{\mathrm{pol}}$.  
In the proof of the first main theorem, we need a stronger result:
For a given positive integer $k$, 
if $S/I$ satisfies Serre's condition $(S_k)$, 
then so does $S^{\mathrm{pol}}/I^{\mathrm{pol}}$; see \cite{MuT}.
Note that a similar statement for (FLC) does not hold in general. 

\par \vspace{2mm}
\subsection{Simplicial join}
Let $\Gamma$ (resp. $\Lambda$) be a non-empty 
simplicial complex on $V_1$ (resp. $V_2$) 
such that $V_1 \cap V_2 = \emptyset$. 
Then the \textit{simplicial join} of $\Gamma$ and $\Lambda$, denoted by 
$\Gamma*\Lambda$, is defined as follows:
\[
 \Gamma * \Lambda = \big\{ F_1 \cup F_2 \;:\; F_1 \in \Gamma, \; F_2 \in \Lambda \big\}. 
\] 
Then  $\Gamma * \Lambda$ is a simplicial complex on $V_1 \cup V_2$ and 
$\mathcal{F}(\Gamma * \Lambda)=\big\{ F_1 \cup F_2 \;:\; F_1 \in 
\mathcal{F}(\Gamma), \; F_2 \in \mathcal{F}(\Lambda) \big\}$. 
In particular, $\dim \Gamma * \Lambda = \dim \Gamma + \dim \Lambda+1$. 
Moreover, the $i$-th reduced homology group $\widetilde{H}_i(\Gamma * \Lambda)$  
over a field $K$ 
of  $\Gamma * \Lambda$ is given by 
the so-called \textit{K\"unneth formula}:
\begin{equation} \label{Kunneth}
\widetilde{H}_i(\Gamma * \Lambda)  \cong   
\bigoplus_{p+q=i-1} \widetilde{H}_p(\Gamma) \otimes  \widetilde{H}_q(\Lambda). 
\end{equation} 
Notice that $K[\Gamma * \Lambda] \cong K[\Gamma] \otimes_K K[\Lambda]$ as $K$-algebras;
see \cite[Lemma 1]{Fr}. 

\par 
For any disjoint union of two graphs $G_1$, $G_2$, 
we have $\Delta(G_1 \coprod G_2) = \Delta(G_1) * \Delta(G_2)$.

\vspace{3mm}
\section{Symbolic powers of edge ideals of disjoint union of complete graphs}

Let $r$, $n_1,\ldots,n_r$ be
positive integers, and let
\[
S = K[x_{ij} \;:\; 1 \le i \le r,\,1 \le j \le n_i]
\]
be a polynomial ring over a field $K$. 
For each integer $i$ with $1 \le i \le r$, if we put 
\[
 P_{ij} = (x_{i1},\ldots,\widehat{x_{ij}},\ldots, x_{in_i})S,
 \quad \text{and} \quad 
 I_i = P_{i1}\cap P_{i2} \cap \cdots \cap P_{in_i},  
\]
then $I_i$ is equal to $I(K_{n_i})S$, where 
\[
 I(K_{n_i})=(x_{ij}x_{ik} \,:\, 1 \le j < k \le n_i)
K[x_{i1}\ldots, x_{in_i}],
\]
denotes the edge ideal of the complete $n_i$-graph $K_{n_i}$ 
on the vertex set $V_i = \{x_{i1},\ldots,x_{in_i}\}$ 
for each $i=1,\ldots,r$. 

\par \vspace{2mm}
Let $G$ be the disjoint union of complete $n_i$-graphs 
for $i=1,2,\ldots,r$:
\[
 G = K_{n_1} \;\textstyle{\coprod} \;K_{n_2}\; 
 \textstyle{\coprod}\; \cdots \;\textstyle{\coprod} \,K_{n_r}.
\] 
Then the edge ideal $I(G)$ of $G$ is equal to $I_1 + I_2 + \cdots + I_r$. 
Moreover, an irredundant primary decomposition of $I(G)$ is given by 
\begin{equation} \label{eq-PrimeDecompo}
 I(G) =  \bigcap_{j_1,\ldots,j_r} (P_{1j_1} + \cdots + P_{rj_r}), 
\end{equation}
where $j_1,\ldots,j_r$ move through the whole range 
$1 \le j_1 \le n_1,\ldots, 1 \le j_r \le n_r$. 
In particular, 
\begin{equation} \label{eq-EdgePower}
  I(G)^{(\ell)} 
= \bigcap_{j_1,\ldots,j_r} (P_{1j_1} + \cdots + P_{rj_r})^{\ell}
\end{equation}
for every integer $\ell \ge 1$. 
If we put  
\[
 x_i = x_{i1} + x_{i2} + \cdots + x_{in_i}
\]
for $i=1,2,\ldots,r$,  
then a sequence $x_1,\ldots,x_r$ forms a system of parameters of $S/I(G)$ 
(and hence $S/I(G)^{(\ell)}$ for every $\ell \ge 1$). 
Thus
\[
 \dim S/I(G)^{(\ell)} = \dim S/I(G) = r. 
\]

\par \vspace{2mm}
The main goal of this section is to prove the following theorem:

\begin{thm} \label{SPcomplete}
Let $S=K[x_{ij} \;:\; 1 \le i \le r,\,1 \le j \le n_i]$ be a polynomial ring 
over a field $K$. 
Let $G$ be a disjoint union of complete $n_i$-graphs$:$ 
$G = K_{n_1} \textstyle{\coprod} K_{n_2} 
\textstyle{\coprod} \cdots \textstyle{\coprod} K_{n_r}$. 
Then $S/I(G)^{(\ell)}$ is Cohen--Macaulay for every $\ell \ge 1$. 
\end{thm}

\begin{remark} \label{rem-PowerRestrict}
In the above theorem, we do not need to assume that $\max\{n_1,\ldots,n_d\} \ge 2$. 
\end{remark}

\par \vspace{1mm}
In order to prove Theorem \ref{SPcomplete}, we need the following key lemma. 

\begin{lemma} \label{Key}
Let $\Gamma$ $($resp. $\Lambda$$)$ be a simplicial complex 
on $V_1$ $($resp. $V_2$$)$
such that $V_1 \cap V_2 = \emptyset$. 
Put $\Delta = \Gamma * \Lambda$ and $V = V_1 \cup V_2$. 
Set $S_1 = K[V_1]$, $S_2 = K[V_2]$ and $S=S_1 \otimes_K S_2$. 
If $S_1/I_{\Gamma}^{(i)}$ and $S_2/I_{\Lambda}^{(i)}$ are Cohen-Macaulay 
for every $i \le \ell$, then $S/I_{\Delta}^{(\ell)}$ is Cohen-Macaulay. 
\end{lemma}

\begin{proof}
We may assume that $V_1=\{1,2,\ldots,m\}$, 
$V_2 = \{m+1,\ldots,n\}$ and $V=[n]$. 
Note that $\Gamma$, $\Lambda$, and $\Delta$ are pure. 
By an inductive argument on $n = \sharp(V)$, we may assume that 
$S/I_{\Delta}^{(\ell)}$ has (FLC). 
Then  we must show that $\widetilde{H}_i(\Delta^{(\ell)}_{\bf a})=0$ for all 
${\bf a} \in \mathbb{N}^n$ and $i < \dim \Delta_{\bf a}^{(\ell)} = \dim \Delta$
with $\Delta_{\bf a}^{(\ell)} \ne \emptyset$. 
We first prove the following claim.
\begin{description}
\item[Claim 1] For each ${\bf a} \in \mathbb{N}^n$, 
$\Delta_{\bf a}^{(\ell)} = \bigcup_{k=1}^{\ell} 
\Gamma_{{\bf a}_1}^{(\ell+1-k)} * \Lambda_{{\bf a}_2}^{(k)}$ holds, 
where ${\bf a}_1 = {\bf a} \,|_{V_1}$ and ${\bf a}_2 = {\bf a}\,|_{V_2}$. 
\end{description}

\par \vspace{2mm}
Since $\mathcal{F}(\Delta) 
= \{F_1 \cup F_2 \,:\, F_1 \in \mathcal{F}(\Gamma),
\,F_2 \in \mathcal{F}(\Lambda)\}$, we have 
\begin{eqnarray*}
\Delta_{\bf a}^{(\ell)} 
&=& \big\langle F_1 \cup F_2 \,:\; F_1 \in \mathcal{F}(\Gamma),\; 
F_2 \in \mathcal{F}(\Lambda),\; 
0 \le \sum_{t \in V \setminus F} a_t \le \ell-1 \big\rangle \\
&=& \bigcup_{k=1}^{\ell}
\big\langle F_1 \cup F_2 \,:\; F_1 \in \mathcal{F}(\Gamma),\; 
F_2 \in \mathcal{F}(\Lambda), \; \sum_{t \in V_1 \setminus F_1} a_t \le \ell-k,\, 
\sum_{t' \in V_2 \setminus F_2} a_{t'} \le k-1 \big\rangle \\
& =& \bigcup_{k=1}^{\ell} \Gamma_{{\bf a}_1}^{(\ell-k+1)} * \Lambda_{{\bf a}_2}^{(k)},  
\end{eqnarray*} 
as required. 
We have proved the claim 1. 

\par \vspace{3mm}
Put  $d=\dim \Delta+1 = \dim \Gamma + \dim \Lambda+2$. 
For $j = 1,\ldots,\ell$, we set 
\[
\Pi_j = \bigcup_{k=1}^{j}  
\Gamma_{{\bf a}_1}^{(\ell-k+1)} * \Lambda_{{\bf a}_2}^{(k)}. 
\]
We next prove the following claim. 

\begin{description}
\item[Claim 2]  $\widetilde{H}_i(\Pi_j) =0$ holds 
for every $i < d-1$ and $j=1,\ldots,\ell$.  
\end{description} 

We use an induction on $j$. 
First consider  the case where $j=1$. 
Then $\Pi_1 = \Gamma_{{\bf a}_1}^{(\ell)} * \Lambda_{{\bf a}_2}$. 
As $I_{\Gamma}^{(\ell)}$ and $I_{\Lambda}$ are Cohen-Macaulay by assumption, 
we get 
\[
p < \dim \Gamma = \dim \Gamma_{{\bf a}_1}^{(\ell)} \Longrightarrow 
\widetilde{H}_p(\Gamma_{{\bf a}_1}^{(\ell)}) =0 
\]
and 
\[
q < \dim \Lambda = \dim \Lambda_{{\bf a}_2} \Longrightarrow 
\widetilde{H}_q (\Lambda_{{\bf a}_2}) =0.  
\]
Now suppose that $i < \dim \Pi_1 = \dim \Gamma + \dim \Lambda+1 = d-1$. 
Then for any pair $(p,q)$ with $p+q=i-1$, either $p < \dim \Gamma$  or 
$q < \dim \Lambda$ holds. 
Hence the K\"unneth formula (see subsection 1.6) yields that 
\[
\widetilde{H}_i (\Pi_1) \cong \bigoplus_{p+q=i-1} \widetilde{H}_p (\Gamma_{{\bf a}_1}^{(\ell)})
\otimes_K \widetilde{H}_q (\Lambda_{{\bf a}_2})=0. 
\]
So we have proved the case where $j=1$. 
\par 
Now assume that $(\ell \ge) j \ge 2$ and $\widetilde{H}_i(\Pi_{j-1})=0$ for all $i < d-1$. 
Then we must show that $\widetilde{H}_i(\Pi_{j})=0$ for all $i < d-1$. 
In order to do that, we put $\Sigma = \Gamma_{{\bf a}_1}^{(\ell-j+1)} * \Lambda_{{\bf a}_2}^{(j)}$. 
Then $\Pi_{j-1} \cup \Sigma = \Pi_j$ and 
\begin{eqnarray*}
\Pi_{j-1} \cap \Sigma 
& = & \bigcup_{k=1}^{j-1} \big(\Gamma_{{\bf a}_1}^{(\ell-k+1)} * \Lambda_{{\bf a}_2}^{(k)} \big) 
\cap \big(\Gamma_{{\bf a}_1}^{(\ell-j+1)} * \Lambda_{{\bf a}_2}^{(j)} \big) \\
& = & \bigcup_{k=1}^{j-1} 
\bigg\{ \big(\Gamma_{{\bf a}_1}^{(\ell-k+1)} * \Lambda_{{\bf a}_2}^{(k)} \big) \cap 
\big(\Gamma_{{\bf a}_1}^{(\ell-j+1)} * \Lambda_{{\bf a}_2}^{(j)} \big)
\bigg\} \\
&=& \bigcup_{k=1}^{j-1} \big(\Gamma_{{\bf a}_1}^{(\ell-j+1)} * \Lambda_{{\bf a}_2}^{(k)} \big) \\
& =& \Gamma_{{\bf a}_1}^{(\ell-j+1)} * \Lambda_{{\bf a}_2}^{(j-1)}. 
\end{eqnarray*}
Thus the Mayer-Vietoris sequence yields the following exact sequence for each $i$:
\[
\cdots \to \widetilde{H}_i (\Pi_{j-1}) \oplus 
\widetilde{H}_i (\Gamma_{{\bf a}_1}^{(\ell-j+1)} * \Lambda_{{\bf a}_2}^{(j)})
\to \widetilde{H}_i (\Pi_j) \to  
\widetilde{H}_{i-1} (\Gamma_{{\bf a}_1}^{(\ell-j+1)} * \Lambda_{{\bf a}_2}^{(j-1)}) \to \cdots .
\]
By a similar argument as above, we have 
\[
\widetilde{H}_i (\Gamma_{{\bf a}_1}^{(\ell-j+1)} * \Lambda_{{\bf a}_2}^{(j)}) = 
\widetilde{H}_{i-1} (\Gamma_{{\bf a}_1}^{(\ell-j+1)} * \Lambda_{{\bf a}_2}^{(j-1)}) =0
\]
for all $i < d-1$. 
Moreover, the induction hypothesis implies  $\widetilde{H}_i(\Pi_{j-1})=0$ for all $i < d-1$. 
Hence $\widetilde{H}_i(\Pi_{j})=0$ for all $i < d-1$, as required. 
Therefore we obtain that $\widetilde{H}_i(\Delta_{\bf a}^{(\ell)}) =0$ for all $i < d-1$. 
This completes the proof.  
\end{proof}

\begin{proof}[Proof of Theorem $\ref{SPcomplete}$]
We prove the assertion by an induction on $r$. 
If $r=1$, then the assertion is clear because $\dim S/I(K_{n_1})^{(\ell)}=\dim S/I(K_{n_1}) =1$. 
\par
Let $r \ge 2$. 
Put 
\[
S'=K[x_{ij} \,:\, 1 \le i \le r-1,\; 1 \le j \le n_i] \quad 
\text{and} \quad 
G'=K_{n_1} \coprod \cdots \coprod K_{n_{r-1}}. 
\]
By the induction hypothesis, we may assume that 
 $S'/I(G')^{(i)}$ is Cohen-Macaulay for all $i \ge 1$. 
As $\Delta(G) = \Delta(G' \coprod K_{n_r}) = \Delta(G')*\Delta(K_{n_r})$, 
by virtue of Lemma \ref{Key}, we can conclude that 
$I(G)^{(\ell)}$ is Cohen-Macaulay for all $\ell \ge 1$.     
\end{proof}

\par \vspace{2mm}
In order to discuss (FLC) properties of symbolic or ordinary powers, 
we generalize Theorem \ref{SPcomplete}  to the following corollary.   

\begin{cor} \label{Main1-extend}
Let $G = K_{n_1}\coprod \ldots \coprod K_{n_r}$ be a disjoint union of finitely many complete graphs, 
and let $y_1,\ldots,y_s$ be variables which are not vertices of $G$. 
Put $S=K[v \,:\, v \in G]$, $T=S[y_1,\ldots,y_s]$ and $I=I(G)+(y_1,\ldots,y_s)$.  
Then $T/I^{(\ell)}$ is Cohen--Macaulay for all $\ell \ge 1$.  
\end{cor}

\begin{proof}
We may assume that $s=1$, and put $y=y_1$ for simplicity. 
Let $I(G)=\cap_j P_j$ be an irredundant primary decomposition of $I(G)$. 
Since $I=\cap_j (P_j,y)$ gives an irredundant primary decomposition of $I$,
we have 
\[
I^{(\ell)}  =  \bigcap_j (P_j,y)^{\ell} 
 =  \bigcap_j \sum_{k=0}^{\ell} P_j^{k} y^{\ell-k} 
 =  \sum_{k=0}^{\ell} \left(\cap_i P_j^{k}\right) y^{\ell-k} 
 =  \sum_{k=0}^{\ell} I(G)^{(k)} y^{\ell-k}.  
\]
Hence it follows that 
\[
T/I^{(\ell)} \cong S/I(G)^{(\ell)} \oplus S/I(G)^{(\ell-1)} \oplus \cdots \oplus S/I(G)
\]
as $S$-modules. 
Since all $S$-modules of the right-hand side are Cohen--Macaulay, so is $T/I^{(\ell)}$, as required. 
\end{proof}

\begin{exam} \label{ex-ci}
If $G$ consists of $r$ isolated edges and $s$ isolated vertices, then 
\[
 S=K[x_{11},x_{12},\ldots,x_{r1},x_{r2},y_1,\ldots,y_s],\quad 
I(G)=(x_{11}x_{12},\ldots,x_{r1}x_{r2}). 
\]
\par 
In particular, $S/I(G)^{(\ell)}=S/I(G)^{\ell}$ is Cohen--Macaulay for every $\ell \ge 1$. 
\par 
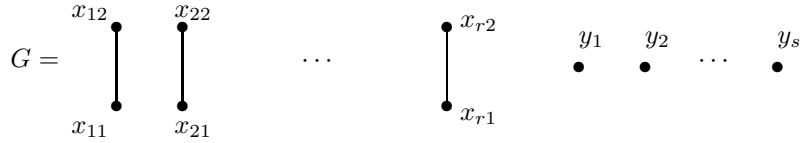
\vspace{4mm}
\begin{center}
\begin{picture}(200,40)
  \put(-40,20){$G=$}
  \put(0,5){\circle*{4}}
  \put(25,5){\circle*{4}}
  \put(125,5){\circle*{4}}
  \put(175,20){\circle*{4}}
  \put(200,20){\circle*{4}}
  \put(220,20){$\cdots$}
  \put(250,20){\circle*{4}}
  \put(0,35){\circle*{4}}
  \put(25,35){\circle*{4}}
  \put(125,35){\circle*{4}}
%
  \put(0,5){\line(0,1){30}}
  \put(25,5){\line(0,1){30}}
  \put(125,5){\line(0,1){30}}  
%
  \put(-17,-5){$x_{11}$}
  \put(22,-5){$x_{21}$}
  \put(130,0){$x_{r1}$}
  \put(-17,39){$x_{12}$}
  \put(22,39){$x_{22}$}
  \put(130,35){$x_{r2}$}
  \put(70,20){$\cdots$}
  \put(175,29){$y_1$}
  \put(200,29){$y_2$}
  \put(250,29){$y_s$}
\end{picture}
\end{center}
\par 
This complete intersection complex is the boundary complex of a simplex or 
an iterated cone of a cross polytope. 
Namely, $I(G)=I_{\Delta(\mathcal{P})}$ holds, where 
$\mathcal{P}$ is the $s$-iterated cone of the cross $r$-polytope. 

\end{exam}

\par \vspace{2mm}
The next example shows that our theorem cannot be generalized 
for mixed symbolic powers. 

\begin{exam} \label{ex-mixedpower}
Let $G$ be a complete $n$-graph. 
Then $I(G) = P_1 \cap \cdots \cap P_n$, 
where $P_i=(x_1,\ldots,\widehat{x_i},\ldots,x_n)$
for each $i=1,\ldots,n$. 
Since $\dim S/I(G) =1$ and $P_i^{a}$ has no embedded primes for any integer $a \ge 1$, 
$S/P_1^{a_1} \cap \cdots \cap P_n^{a_n}$ 
is a Cohen--Macaulay ring of dimension $1$ for every positive integers $a_1,\ldots,a_n$. 
\par 
A similar assertion does \textit{not} hold in general for two disjoint union of 
complete graphs. 
For example, let $I(G) = (x_1x_2,x_1x_3,x_2x_3,y_1y_2)$ be the edge ideal of 
$K_3 \;\textstyle{\coprod}\; K_2$ in $S=\QQ[x_1,x_2,x_3,y_1,y_2]$. 
Then 
\[
 I(G) = (x_1,x_2,y_1)\cap (x_1,x_2,y_2)\cap (x_1,x_3,y_1)
\cap (x_1,x_3,y_2)\cap (x_2,x_3,y_1)\cap (x_2,x_3,y_2). 
\]  
Our theorem says that 
\[
 I(G)^{(2)} = (x_1,x_2,y_1)^2 \cap (x_1,x_2,y_2)^2 \cap (x_1,x_3,y_1)^2
\cap (x_1,x_3,y_2)^2 \cap (x_2,x_3,y_1)^2 \cap (x_2,x_3,y_2)^2
\]
is a Cohen--Macaulay ring of dimension $2$, that is, $\pd_S S/I(G)^{(2)}=3$.  
Indeed, by Macaulay 2, the minimal free resolution of 
$S/I(G)^{(2)}$ over $S$ is given by 
\[
0 \to S^5 \to S^{12} \to S^8 \to S \to S/I(G)^{(2)} \to 0.  
\]
However, this is no longer true for mixed symbolic powers. 
For instance, put 
\[
 J_a= (x_1,x_2,y_1)^2 \cap (x_1,x_2,y_2)^2 \cap (x_1,x_3,y_1)^2
\cap (x_1,x_3,y_2)^2 \cap (x_2,x_3,y_1)^2 \cap (x_2,x_3,y_2)^a
\]
for every positive integer $a \ge 2$. 
When $a \le 3$, $S/J_a$ is Cohen--Macaulay. But $S/J_4$ is \textit{not}.  
\end{exam}

\par 
The following question seems to be interesting. 

\begin{quest} \label{Q-MixedPower}
We use the same notation as in $(\ref{eq-PrimeDecompo})$. 
Let $\ell_{j_1,\ldots,j_r}$ be given integers. 
When is the following mixed symbolic power ideal  
\[
\bigcap_{j_1,\ldots,j_r} (P_{1,j_1}+\cdots + P_{r,j_r})^{\ell_{j_1,\ldots,j_r}}
\]
Cohen--Macaulay?
\end{quest}

\medskip
\section{Non-Cohen--Macaulayness of symbolic powers}

\subsection{Cohen--Macaulay properties of symbolic powers}
In the previous section, we proved that 
all symbolic powers of the edge ideal of a disjoint union of 
finitely many complete graphs are Cohen--Macaulay.  
In this section, we prove the converse. 
That is, the main purpose of this section is to prove Theorem \ref{Serre2}. 
Using these results, we prove the first main theorem.   
Moreover, as an application, we also prove an improvement of the main theorem 
\cite{CRTY} with respect to Cohen--Macaulayness of ordinary powers. 

\begin{thm} \label{Serre2}
Let $G$ be a graph which is not a disjoint union of finitely many complete graphs. 
Then for any $\ell \ge 3$, 
$S/I(G)^{(\ell)}$ does not satisfy Serre's condition $(S_2)$. 
\end{thm}

\begin{remark}
The assumption that $\ell \ge 3$ is essential. 
For example, let $G$ be a pentagon, and set $I(G) = (x_1x_2,x_2x_3,x_3x_4,x_4x_5,x_5x_1)$ 
in $S=K[x_1,x_2,x_3,x_4,x_5]$. 
Then $I(G)$ is not complete intersection, but $S/I(G)^{(2)} = S/I(G)^2$ is Cohen--Macaulay. 
\end{remark}

\par \vspace{2mm}
In order to study Cohen--Macaulayness of 
higher symbolic powers of edge ideals, we use the notion of polarization. 
Let $I$ be a monomial ideal of $S$, 
and let $I^{\mathrm{pol}} \subseteq S^{\mathrm{pol}}$ denote the polarization 
of $I$. 

\par 
Let $G$ be a graph with vertex set $V(G)=\{x_1,x_2,\ldots,x_n\}$. 
Let $\Delta=\Delta(G)$ be the complementary 
simplicial complex of $G$.  
For a positive integer $\ell$, let $\Delta^{(\ell)}$ be the simplicial complex 
such that $I_{\Delta^{(\ell)}}=(I(G)^{(\ell)})^{\mathrm{pol}}$. 

\par 
For a positive integer $\ell$ and for any fixed $i$, 
we put $(x_i^{\ell})^{\mathrm{pol}} 
=x_i^{(0)}x_i^{(1)}\cdots x_i^{(\ell-1)}$, 
where $x_i^{(0)} = x_i$. 
Furthermore, we put 
$(x_1^{\ell_1}\cdots x_n^{\ell_n})^{\mathrm{pol}}
=(x_1^{\ell_1})^{\mathrm{pol}}\cdots (x_n^{\ell_n})^{\mathrm{pol}}$.
See Section 2 for more details.   
In order to study facets of $\Delta^{(\ell)}$, we need the following lemma. 

\begin{lemma} \label{pol-lemma}
Under the above notation, we have 
\[
((x_1,\ldots,x_h)^{\ell})^{\mathrm{pol}} 
=
\bigcap_{
i_1+\cdots + i_h \le \ell-1} 
(x_1^{(i_1)},\ldots,x_{h}^{(i_h)}).
\]
\end{lemma}

\begin{proof}
By the definition of polarization, we have 
\begin{eqnarray*}
((x_1,\ldots,x_h)^{\ell})^{\mathrm{pol}} 
& = & \big(
x_1^{j_1}\cdots x_h^{j_h} \; :\; j_1,,\ldots,j_h \ge 0,\; 
j_1+\cdots +j_h = \ell \big)^{\mathrm{pol}} \\
& = & \big(\prod_{k=1}^h x_k^{(0)} x_k^{(1)} \cdots x_k^{(j_k-1)} 
\;:\; j_1,\ldots,j_h \ge 0,\; j_1 + \cdots + j_h = \ell \big). 
\end{eqnarray*}
So, in order to obtain the required primary decomposition, it suffices to show that 
\[
((x_1,\ldots,x_h)^{\ell})^{\mathrm{pol}} \subseteq (x_1^{(i_1)},\ldots,x_{h}^{(i_h)})
\Longleftrightarrow 
i_1 +  \cdots + i_h \le \ell-1.
\]
Suppose $i_1 +  \cdots + i_h \ge \ell$. 
Take a monomial $M = \prod_{k=1}^h x_k^{(0)}x_k^{(1)}\cdots x_k^{(i_k-1)}$. 
Then it is clear that $M \notin (x_1 ^{(i_1)},\ldots,x_{h}^{(i_h)})$.
On the other hand, $M$ is contained in $((x_1,\ldots,x_h)^{\ell})^{\mathrm{pol}}$
because there exists a sequence $(j_1,\ldots,j_h)$ such that 
$0 \le j_k \le i_k$ for each $k$ and $j_1 + \cdots + j_h = \ell$.  
\par 
Next suppose that $i_1 +  \cdots + i_h \le \ell-1$. 
If $((x_1,\ldots,x_h)^{\ell})^{\mathrm{pol}} \not \subseteq (x_1^{(i_1)},\ldots,x_{h}^{(i_h)})$, 
then there exists a monomial 
$M = \prod_{k=0}^h x_k^{(0)}\cdots x_k^{(j_k-1)}$ with $j_1 + \cdots + j_h=\ell$ 
such that $M$ is not contained in 
$(x_1^{(i_1)},\ldots,x_{h}^{(i_h)})$. 
Hence $j_k \le i_k$ for each $k$. 
But $\ell=j_1 + \cdots + j_h \le i_1+\cdots+ i_h \le \ell-1$. This is a contradiction.  
\end{proof}

\par 
By the above lemma, we get the following corollary. 

\begin{cor} \label{facets-ell}
Under the above notation, we set $V^{(i)} = \{x_1^{(i)},\ldots,x_n^{(i)}\}$ for 
each $i=1,2,\ldots,\ell-1$. 
Then $\mathcal{F}(\Delta^{(\ell)})$ consists of 
the following subsets of $V \cup V^{(1)} \cup \cdots \cup V^{(\ell-1)}:$
\begin{eqnarray*}
&& \big(F \cup \big\{x_{i_{1,1}},\ldots,x_{i_{1,j_1}},x_{i_{2,1}},\ldots,x_{i_{2,j_2}},\ldots,
x_{i_{\ell-1,1}},\ldots,x_{i_{\ell-1, j_{\ell-1}}}\big\}\big) \\
&& \cup \big(V^{(1)} \setminus \big\{x_{i_{1,1}}^{(1)},\ldots,x_{i_{1,j_1}}^{(1)} \big\}\big) 
\cup \cdots \cup 
\big(V^{(\ell-1)} \setminus 
\big\{x_{i_{\ell-1,1}}^{(\ell-1)},\ldots,x_{i_{\ell-1,j_{\ell-1}}}^{(\ell-1)} \big\}\big),  
\end{eqnarray*}
where $F$ and $x_{i}$'s run through
\[
\begin{array}{cl}
\bullet & F \in \mathcal{F}(\Delta); \\
\bullet & 0 \le j_1,j_2,\ldots,j_{\ell-1} \le n, \quad j_1 + 2 j_2 + \cdots + (\ell-1)j_{\ell-1} \le \ell-1; \\
\bullet & \big\{x_{i_{1,1}},\ldots,x_{i_{\ell-1,j_{\ell-1}}}\big\} \cap F = \emptyset,\quad
\sharp \big\{x_{i_{1,1}},\ldots,x_{i_{\ell-1,j_{\ell-1}}}\big\} 
= j_1 + j_2 + \cdots + j_{\ell-1}. 
\end{array}
\]
\par 
In particular, if $\Delta$ is pure, then so is $\Delta^{(\ell)}$. 
\end{cor}

\begin{proof}
By definition, we have 
\[
 I_{\Delta^{(\ell)}} 
= \big((I(G))^{(\ell)}\big)^{\mathrm{pol}}
= \big(\bigcap_{F \in \mathcal{F}(\Delta)}\!\! P_F^{\ell}\big)^{\mathrm{pol}}
= \bigcap_{F \in \mathcal{F}(\Delta)}\!\! (P_F^{\ell})^{\mathrm{pol}}.
\]
If $P_F=(y_1,\ldots,y_h)$, then
\[
(P_F^{\ell})^{\mathrm{pol}} = \bigcap_{i_1+\cdots+i_h \le \ell-1} \!\! 
(y_{1}^{(i_1)},\ldots,y_h^{(i_h)})  
\]
by the above lemma. 
\par
Let $G \in \mathcal{F}(\Delta^{(\ell)})$. 
Then there exist 
a facet $F \in \mathcal{F}(\Delta)$ and 
integers $0 \le i_1 \le \ldots \le i_h$ 
with $i_1+\cdots + i_h \le \ell-1$ 
such that 
\begin{eqnarray*}
V \cup V^{(1)} \cup \cdots \cup V^{(\ell-1)} \setminus G 
&=& \{y_1^{(i_1)},\ldots,y_h^{(i_h)}\}, \\
V \setminus F &=& \{y_1,\ldots,y_h\}. 
\end{eqnarray*}
Putting
\[
\big\{y_1^{(i_1)},\ldots,y_h^{(i_h)}\big\} = 
\{x_{i_{0,1}}^{(0)},\ldots,x_{i_{0,j_0}}^{(0)},\ldots,
x_{i_{\ell-1,1}}^{(\ell-1)},\ldots, x_{i_{\ell-1,j_{\ell-1}}}^{(\ell-1)}\},
\]
we get a required form of $G$.  
\end{proof}

\begin{proof}[Proof of Theorem $\ref{Serre2}$]
Assume that $S/I(G)^{(\ell)}$ satisfies $(S_2)$. 
As $I(G)=\sqrt{I(G)^{(\ell)}}$, $S/I(G)$ also satisfies $(S_2)$ by \cite{HTT}. 
In particular, $I(G)$ is pure. 
Since some connected component of $G$ is not a complete graph by assumption, 
there exist $x_1,x_2,x_3 \in V(G)$ such that 
\[
 \{x_1, x_2\} ,\,\{x_1, x_3\} \in E(G), \quad \text{and} \; \{x_2,x_3\} \notin E(G). 
\]
We may assume that 
$V(G) = \{x_1,x_2,x_3,\ldots,x_n\}$, the vertex set of $G$ 
by renumbering if necessary.  
Let $\Delta=\Delta(G)$ be the complementary 
simplicial complex of $G$, and let $\Delta^{(\ell)}$ be the simplicial complex 
defined as above. 
Set $\widetilde{V} = V \cup V^{(1)} \cup \cdots \cup V^{(\ell-1)}$. 
Note that $\Delta^{(\ell)}$ is a pure simplicial complex on $\widetilde{V}$.  
\par \vspace{2mm}
Now consider the following subset of $\widetilde{V}:$  
\[
F_0 = \left\{
\begin{array}{ccccccc}
x_1, & x_1^{(1)}, & x_1^{(2)}, & \ldots & x_1^{(\ell-3)}, & x_1^{(\ell-2)}, & \square \\[2mm] 
x_2, & x_2^{(1)}, & x_2^{(2)}, & \ldots & x_2^{(\ell-3)}, & \square & x_2^{(\ell-1)}, \\[2mm] 
x_3, & \square & x_3^{(2)}, & \ldots & x_3^{(\ell-3)}, & x_3^{(\ell-2)}, & x_3^{(\ell-1)}, \\[2mm] 
\square & x_4^{(1)}, & x_4^{(2)}, & \ldots & x_4^{(\ell-3)}, & x_4^{(\ell-2)}, & x_4^{(\ell-1)},\\[2mm]
\vdots & \vdots & \vdots & \cdots & \vdots & \vdots & \vdots  \\[2mm] 
\square & x_n^{(1)}, & x_n^{(2)}, & \ldots & x_n^{(\ell-3)}, & x_n^{(\ell-2)}, & x_n^{(\ell-1)} 
\end{array}
\right\}. 
\] 
Then $F_0$ is a face of $\Delta^{(\ell)}$. 
Indeed, we can take a facet $F \in \mathcal{F}(\Delta)$ such that $\{x_2,x_3\} \subseteq F$. 
Since $x_1 \notin F$, 
\[
F' = (F \cup \{x_1\}) \cup V^{(1)} \cup V^{(2)} \cup \cdots \cup V^{(\ell-2)} \cup 
(V^{(\ell-1)} \setminus \{x_1^{(\ell-1)}\})
\]
is a facet of $\mathcal{F}(\Delta^{(\ell)})$ by Corollary \ref{facets-ell}. 
This implies that $F_0 \in \Delta^{(\ell)}$ because $F_0 \subseteq F'$.

\par 
We first prove the following claim: 
\begin{description}
\item[Claim] Any facet of $\link_{\Delta^{(\ell)}}(F_0)$ is given by 
\[
\begin{array}{cl}
& (F \setminus \{x_2,x_3\}) \cup \{x_3^{(1)}\} \cup \{x_2^{(\ell-2)}\}, 
\text{where}
\, F \in \FF(\Delta) \; \text{and} \; \{x_2,x_3\} \subseteq F; \\ 
\text{or} & \\
& (F \setminus \{x_1\}) \cup \{x_1^{(\ell-1)}\}, 
\text{where}\; 
 F \in \FF(\Delta) \; \text{and} \; x_1 \in F.  
\end{array}
\]
\end{description}
In order to prove the claim, it suffices to determine $\FF(\star_{\Delta^{(\ell)}}(F_0))$
because any facet $G$ of $\link_{\Delta^{(\ell)}}(F_0)$ can be written as 
$G = \widetilde{F} \setminus F_0$ 
for some $\widetilde{F} \in \FF(\star_{\Delta^{(\ell)}}(F_0))$. 
\par 
Let $\widetilde{F} \in \FF(\star_{\Delta^{(\ell)}}(F_0))$. 
Then $\widetilde{F} \in \FF(\Delta^{(\ell)})$ and 
$\widetilde{F} \supseteq F_0$. 
In particular, $x_i^{(1)},\ldots,x_i^{(\ell-1)} \in \widetilde{F}$ 
for each $i=4,\ldots,n$ and 
\[
 W_1 = V^{(1)} \setminus \{x_3^{(1)}\},
\quad  
 W_{\ell-2}=V^{(\ell-2)} \setminus \{x_2^{(\ell-2)}\},
\quad  
 W_{\ell-1}= V^{(\ell-1)} \setminus \{x_1^{(\ell-1)}\} \subseteq \widetilde{F}.
\] 
Hence $\widetilde{F}$ is given by one of the following complexes:
{\small
\[
\begin{array}{cclllllllll}
\widetilde{F}_1 &=& (F \cup \{x_1\}) & \cup & V^{(1)} & \cup & V^{(2)}  \cup \cdots \cup V^{(\ell-3)} & 
 \cup & V^{(\ell-2)} & \cup & W_{\ell-1}, \\[1mm]
\widetilde{F}_2 &=& (F \cup \{x_2\}) & \cup & V^{(1)} & \cup & V^{(2)}  \cup \cdots \cup V^{(\ell-3)} & 
 \cup & W_{\ell-2} & \cup & V^{(\ell-1)}, \\[1mm]
\widetilde{F}_3 &=& (F \cup \{x_3\}) & \cup & W_1 & \cup & V^{(2)}  \cup \cdots \cup V^{(\ell-3)} & 
 \cup & V^{(\ell-2)} & \cup & V^{(\ell-1)}, \\[1mm]
\widetilde{F}_{12} &=& (F \cup \{x_1,x_2\}) & \cup & V^{(1)} & \cup & V^{(2)}  \cup \cdots \cup V^{(\ell-3)} & 
 \cup & W_{\ell-2} & \cup & W_{\ell-1}, \\[1mm]
\widetilde{F}_{13} &=& (F \cup \{x_1,x_3\}) & \cup & W_{1} & \cup & V^{(2)}  \cup \cdots \cup V^{(\ell-3)} & 
 \cup & V^{(\ell-2)} & \cup & W_{\ell-1}, \\[1mm]
\widetilde{F}_{23} &=& (F \cup \{x_2,x_3\}) & \cup & W_{1} & \cup & V^{(2)}  \cup \cdots \cup V^{(\ell-3)} & 
 \cup & W_{\ell-2} & \cup & V^{(\ell-1)}. 
\end{array}
\]
}
Now suppose that $\widetilde{F} = \widetilde{F}_2$. 
Then we have $x_1,x_3 \in F$. 
This implies that $x_3 \notin P_F$. 
Hence, $x_1x_3 \in I(G)$ yields $x_1 \in P_F$. 
This contradicts $x_1 \in F$. 
Therefore it does not occur that $\widetilde{F} = \widetilde{F}_2$. 
Similarly, we have $\widetilde{F} \ne \widetilde{F}_3$. 
\par 
Next suppose that $\widetilde{F} = \widetilde{F}_{12}$.
Then $(\ell-2) j_{\ell-2} + (\ell-1) j_{\ell-1} \ge 2\ell-3 \ge \ell$ because $\ell \ge 3$. 
This is impossible. 
Hence $\widetilde{F} \ne \widetilde{F}_{12}$. 
Similarly, we have $\widetilde{F} \ne \widetilde{F}_{13}$. 
Consequently, either 
\[
\widetilde{F} = \widetilde{F}_1 \quad \text{and}\quad x_2,x_3 \in F,\quad x_1 \notin F  
\] 
or 
\[
\widetilde{F} = \widetilde{F}_{23} \quad \text{and}\quad x_1 \in F,
\quad x_2,x_3 \notin F 
\]
holds. 
In other words, any $G \in \FF(\link_{\Delta^{(\ell)}}(F_0))$ can be written as 
\[
G' = (F \setminus \{x_2,x_3\}) \cup \{x_3^{(1)}\} \cup \{x_2^{(\ell-2)}\}  
\]
for some $F \in \FF(\Delta)$ such that $x_1 \notin F$ and $x_2,\,x_3 \in F$; 
or 
\[
G'' = (F \setminus \{x_1\}) \cup \{x_1^{(\ell-1)}\}  
\]
for some $F \in \FF(\Delta)$ such that $x_1 \in F$ and $x_2,x_3 \notin F$.  
So, we proved the claim.  
\par \vspace{2mm}
Choose $G'$ and $G''$ of the above type, respectively.  
Note that there exist those facets as $(x_1x_2,x_1x_3) \subseteq I_{\Delta}$. 
Then one can find no chain of facets in $\link_{\Delta^{(\ell)}}(F_0)$ such that 
\[
 G'=G_0,\,G_1, \ldots, G_r = G''
\]
with $\sharp(G_i \cap G_{i-1}) = d-1$, where 
$d = \dim K[\link_{\Delta^{(\ell)}}(F_0)]$ since 
both $x_3^{(1)}$ and $x_2^{(\ell-2)}$ are 
contained in $G'$ but not in $G''$. 
Thus $\link_{\Delta^{(\ell)}}(F_0)$ is \textit{not} connected in codimension $1$, 
and hence it does \textit{not} satisfy $(S_2)$.  
By the lemma below, we can conclude that $S/I(G)^{(\ell)}$ does \textit{not} 
satisfy $(S_2)$, as required.  
\end{proof}

\par 
The following lemma was used in the proof of \cite[Theorem 4.1]{MuT}. 
Moreover, it is clear that $S/I$ is Cohen--Macaulay if and only if 
so is $S^{\mathrm{pol}}/I^{\mathrm{pol}}$ because 
$S/I$ is isomorphic to a quotient of $S^{\mathrm{pol}}/I^{\mathrm{pol}}$
by a regular sequence. 

\begin{lemma}[\textrm{See the proof of \cite[Theorem 4.1]{MuT}}]
Let $k \ge 1$ be any integer. 
Let $I \subseteq S$ be a monomial ideal, 
and let $I^{\mathrm{pol}} \subseteq S^{\mathrm{pol}}$ 
denote the polarization of $I$. 
If $S/I$ satisfies $(S_k)$, then so does $S^{\mathrm{pol}}/I^{\mathrm{pol}}$. 
\end{lemma}

\par 
We are now ready to prove the first main theorem in this paper. 

\begin{thm} \label{Main-CM}
Let $I(G)\subseteq S$ be the edge ideal of a graph $G$. 
Then the following conditions are equivalent$:$ 
\begin{enumerate}
 \item $S/I(G)^{(\ell)}$ is Cohen--Macaulay for every integer $\ell \ge 1$. 
 \item $S/I(G)^{(\ell)}$ is Cohen--Macaulay for some $\ell \ge 3$. 
 \item $S/I(G)^{(\ell)}$ satisfies Serre's condition $(S_2)$ for some $\ell \ge 3$. 
 \item $G$ is a disjoint union of finitely many complete graphs. 
\end{enumerate}   
\end{thm}

\begin{proof}
Let $I(G) \subseteq S$ be the edge ideal with $\dim S/I(G) \ge 2$. 
\par \vspace{2mm} 
$(1) \Longrightarrow (2):$ 
This is clear. 
\par \vspace{2mm} 
$(2) \Longrightarrow (3):$ 
Since any Cohen--Macaulay ring satisfies 
Serre's condition $(S_2)$, it is clear. 
\par \vspace{2mm} 
$(3) \Longrightarrow (4):$
Now suppose that $G$ cannot be written as a disjoint union 
of finitely many complete graphs.   
Then, for any $\ell \ge 3$, $S/I(G)^{(\ell)}$ does not satisfy $(S_2)$ 
by Theorem \ref{Serre2}. This contradicts the assumption. 
\par \vspace{2mm} 
$(4) \Longrightarrow (1):$ 
By Theorem \ref{SPcomplete}, if $G$ is a disjoint union 
of finitely many complete graphs, then $S/I(G)^{(\ell)}$ is Cohen--Macaulay 
for every $\ell \ge 1$.  
\end{proof}

\begin{remark}
By a similar argument as in Corollary \ref{Main1-extend}, we can generalize 
the above theorem to the case where $I$ contains variables. 
Moreover, in this case, we can replace $S$ with $S[t]$, where $t$ is an indeterminate. 
\end{remark}

\par \vspace{2mm}
\subsection{Cohen--Macaulay properties of ordinary powers}
\par 
Using Theorem \ref{Main-CM}, we can give an improvement 
of the main theorem in \cite{CRTY}. 

\begin{thm}[\textrm{cf. \cite[Theorem 2.1]{CRTY}}] \label{Power-cor}
Let $I(G)$ be the edge ideal of a graph $G$. 
If $S/I(G)^{\ell}$ is Cohen--Macaulay for some $\ell \ge 3$, 
then $I(G)$ is complete intersection. 
\end{thm}

\begin{remark}
In \cite{CRTY}, the authors proved an analogous theorem: 
$I(G)$ is complete intersection whenever 
$S/I(G)^{\ell}$ is Cohen--Macaulay for some $\ell \ge \height I(G)$.    
Note that it is not difficult to derive this from Theorem \ref{Power-cor}.  
\end{remark}

\par 
In order to prove the theorem, we need the following lemma. 

\begin{lemma}[\textrm{See also \cite[Lemma 5.8, Theorem 5.9]{SVV}}] \label{SVV} 
Let $I(G)$ be the edge ideal of a graph $G$. 
Let $t \ge 2$ be an integer.
Then the following conditions are equivalent. 
\begin{enumerate}
 \item $G$ contains no odd cycles of length $2s-1$ for any $2 \le s \le t$. 
 \item $I(G)^{(t)}=I(G)^t$ holds. 
\end{enumerate}
\end{lemma}

\begin{proof} 
Put $I=I(G)$ for simplicity. 
\par 
$(1) \Longrightarrow (2):$ 
It follows from a similar argument as in the proof of 
\cite[Lemma 5.8, Theorem 5.9]{SVV}. 
But for the convenience of the readers, we give a sketch of the proof. 
It is enough to show that 
$\frm \notin \Ass_S(S/I^t)$ if $\dim S/I \ge 1$. 
Now suppose \textit{not}. 
Then we can take a monomial $M \notin I^t$ such that $I^t \colon M = \frm$. 
Since $\depth S/I \ge 1$, we get $M \in I$. 
So we can write $M = x_1x_2 L$ for some $x_1x_2 \in G(I)$ and a monomial $L$. 
By definition, we have $x_2M = x_1x_2^2L \in I^t$. 
It follows that $x_2^2L \in I^{t-1}$ because $I$ is generated by squarefree monomials. 
This yields $M \in x_1 I^{t-1} \cap (I^t :x_1)$. 
\par 
On the other hand, by a similar argument as in the proof of \cite[Lemma 5.8]{SVV}, 
we can show that $x I^m \cap (I^{m+1} \colon x) \subseteq I^{m+1}$ 
for any vertex $x$ and for all $0 \le m \le t-1$ using $(1)$ 
(Notice that there exists a small gap in the final step of the 
proof of \cite[Lemma 5.8]{SVV}.
That is, we obtain an odd cycle if only if $i$ is even.).
In particular, $M \in x_1 I^{t-1} \cap (I^t :x_1)\subseteq I^t$, which 
contradicts the choice of $M$. 
\par 
$(2) \Longrightarrow (1):$
Suppose that $G$ contains an odd cycle of length $2s-1$ with $2 \le s \le t$; 
say, $x_1x_2$, $x_2x_3, \ldots, x_{2s-2}x_{2s-1}$, $x_{2s-1}x_1$. 
Put $M = x_1x_2\cdots x_{2s-1}$. 
Then we show $M(x_1x_2)^{t-s} \in I^{(t)} \setminus I^t$. 
Let $P$ be any associated prime ideal of $I$. 
Then since $P$ is prime and $x_1x_2,x_2x_3,\ldots,x_{2s-2}x_{2s-1},x_{2s-1}x_1 \in P$, 
we get $\sharp(P \cap \{x_1,x_2,\ldots,x_{2s-1}\}) \ge s$. 
Hence $M \in P^{s}$ and thus $M(x_1x_2)^{t-s} \in I^{(t)}$. 
On the other hand, $M (x_1x_2)^{t-s} \notin I^t$ 
because $\deg M(x_1x_2)^{t-s} = 2t-1 < 2t=\indeg I^t$, 
where $\indeg I^t = \min\{m \in \mathbb{Z} \,:\, [I^t]_m \ne 0 \}$. 
\end{proof}

\begin{cor} \label{power-not}
Let $G$ be a disjoint union of complete graphs 
$K_{n_1},\ldots,K_{n_r}$. 
\par 
If $\max\{n_1,\ldots,n_r\} \ge 3$, 
then $I(G)^{(\ell)} \ne I(G)^{\ell}$ for every $\ell \ge 2$. 
In particular, $I(G)^{\ell}$ is not a Cohen--Macaulay ideal.   
\end{cor}

\begin{proof}
Under the assumption, $G$ always contains a triangle ($3$-cycle).  
\end{proof}

\begin{proof}[Proof of Theorem $\ref{Power-cor}$]
Now suppose that $S/I(G)^{\ell}$ is Cohen--Macaulay 
for some integer $\ell \ge 3$, 
and that $I(G)$ is \textit{not} complete intersection. 
\par  
By Theorem \ref{Main-CM}, $G$ can be written as a disjoint union 
of finitely many complete graphs. 
However, this contradicts the above corollary. 
\end{proof}

\par 
The next example shows that the Cohen--Macaulayness of symbolic power ideals 
is different from that of ordinary power ideals.  

\begin{exam} \label{union-triangles}
Let $G$ be a disjoint union of $d$ complete $3$-graphs.   
Set 
\[
I=I(G)=(x_{11}x_{12},x_{11}x_{13},x_{12}x_{13},\ldots, x_{d1}x_{d2},x_{d1}x_{d3},x_{d2}x_{d3})
\] 
in a polynomial ring $S=K[x_{11},x_{12},x_{13},\ldots,x_{d1},x_{d2},x_{d3}]$. 
Then 
\begin{enumerate}
 \item $S/I^{(\ell)}$ is Cohen--Macaulay of dimension $d$ for every $\ell \ge 1$.  
 \item $S/I^{\ell}$ is \textit{not} Cohen--Macaulay for any $\ell \ge 2$. 
 \item $I$ is \textit{not} complete intersection. 
\end{enumerate}
\end{exam}

\begin{proof}
(1) follows from Theorem \ref{Main-CM}. 
\par 
(2) If $\ell \ge 3$, then the assertion follows from Theorem \ref{Power-cor}. 
When $\ell=2$, it follows from the fact $x_{11}x_{12}x_{13} \in I^{(2)} \setminus I^2$.      
\end{proof}

\par \vspace{3mm}
\subsection{Some related results}
In the final of this section, we comment a relationship between our results 
and the theorem by Minh--Trung \cite{MiT}.
Minh and Trung studied Cohen--Macaulay properties of the symbolic power ideals 
for $1$-dimensional simplicial complexes. 

\begin{thm-q}[Minh--Trung; see \cite{MiT}]
Let $\ell \ge 3$ be an integer. 
Let $I=I_{\Delta}$ be the Stanley--Reisner ideal 
of a simplicial complexes of dimension $1$. 
Then $S/I^{(\ell)}$ is Cohen--Macaulay if and only if  
every pair of disjoint edges of $\Delta$ is contained in a cycle of length $4$.  
\end{thm-q}

\par 
If $I_{\Delta}$ is generated by degree $2$ monomials, 
the ideal $I_{\Delta}$ can be regarded as the edge ideal of a graph $G$. 
Then the required condition in the above theorem says that 
$G$ is a disjoint union of two complete graphs. 
So, their theorem does not conflict our theorem. 

\medskip
\section{Finite local cohomology and symbolic power}

\par 
In \cite{GT}, Goto and Takayama introduced the notion of generalized 
complete intersection complex. 
On the other hand, in \cite{TY}, the last two authors defined the notion of 
locally complete intersection complex and gave a structure theorem 
for those complexes.  
Note that $\Delta$ is a generalized complete intersection complex  
if and only if $\Delta$ is a pure, locally complete intersection complex. 

\begin{defn}[\textrm{cf. \cite{TY}}]
Let $\Delta$ be a simplicial complex on the vertex set $V$. 
The complex $\Delta$ is called a \textit{locally complete intersection 
complex} if $K[\link_{\Delta} \{v\}]$ is complete intersection for every 
vertex $v \in V$.  
\end{defn}

\par 
The following result gives a structure theorem for locally 
complete intersection complexes. 

\begin{thm}[\textrm{cf. \cite{TY}}] \label{Structure}
Let $\Delta$ be a simplicial complex on $V$ such that $V \ne \emptyset$. 
Then $\Delta$ is a locally complete intersection complex if and only if 
it is a finitely many disjoint union of the following connected 
complexes$:$ 
\begin{enumerate}
\item[(a)] a complete intersection complex $\Gamma$ with $\dim \Gamma \ge 2;$ 
\item[(b)] $m$-gon $(m \ge 3);$
\item[(c)] $m'$-pointed path $(m' \ge 2);$ 
\item[(d)] a point. 
\end{enumerate}
When this is the case, $K[\Delta]$ is Cohen--Macaulay 
$($resp., Buchsbaum $)$ 
if and only if $\dim \Delta =0$ or $\Delta$ is connected  $($resp., pure$)$. 
\end{thm} 
 
\par
Moreover, for any pure simplicial complex $\Delta$, 
it is a locally complete intersection complex if and only if  
$S/I_{\Delta}^{\ell}$ has (FLC) for all $\ell \ge 1$
(or, more generally, for infinitely many $\ell \ge 1$). 
But, for a fixed $\ell \ge 1$, it is open when $S/I^{\ell}$ has (FLC). 

\subsection{FLC properties of symbolic powers}
 In this section, we consider the following question, which is closely related to the 
above question in the case of edge ideals.  

\begin{quest} \label{FLC-quest}
Let $I(G)$ be denote the edge ideal of a graph $G$.  
Let $\ell \ge 1$ be an integer. 
When does $S/I(G)^{(\ell)}$ have $(FLC)$? 
\end{quest}

\par 
As one of answers to this question, we prove the second main theorem (Theorem \ref{Main-FLC}). 
We first  prove the following proposition. 

\begin{prop} \label{Main2-flc}
Let $\Delta_{n_1,\ldots,n_r}$ denote the simplicial complex 
whose Stanley--Reisner ideal is equal to 
the edge ideal of a disjoint union of 
complete graphs $K_{n_1},\ldots,K_{n_r}$. 
That is, 
\[
 I_{\Delta_{n_1,\ldots,n_r}} = 
I(K_{n_1} \textstyle{\coprod} \cdots \textstyle{\coprod} K_{n_r}). 
\]  
Let $\Delta$ be a simplicial complex defined by 
\[
 \Delta = \Delta_{n_{11},\ldots,n_{1d}} \;\textstyle{\coprod}\; 
\Delta_{n_{21},\ldots,n_{2d}} \;\textstyle{\coprod} 
\; \ldots \;
\textstyle{\coprod} \;
\Delta_{n_{p1},\ldots,n_{pd}}, 
\]
where one can take all $n_{ij}=1$ when $p \ge 2$. 
Put
\[
S=K \left[x_{ij}^{(k)} \;:\;  
1 \le i \le d; \;
\;1 \le k \le p; \;
1 \le j \le n_{ki}
\right],
\]
a polynomial ring over $K$, and 
\begin{eqnarray*}
 I_{\Delta} &=& \left(x_{ij}^{(k)} x_{ij'}^{(k)} \;:\; 1 \le i \le d,\; 
1 \le j < j' \le n_{ki}; \;
1 \le k \le p  \right)S \\
& + & (x_{ij}^{(k)}x_{i'j'}^{(m)} \;:\; 1 \le i,\,i' \le d,\; 
1 \le j \le n_{ki}, 1 \le j' \le n_{mi'},\; 1 \le k < m \le p)S. 
\end{eqnarray*}
Then $S/I_{\Delta}^{(\ell)}$ has $($FLC$)$ for every $\ell \ge 1$. 
\end{prop}

\begin{proof}
Put $I=I_{\Delta}$. 
Since $\dim \Delta_{n_1,\ldots,n_d} =d-1$, $\Delta$ is a 
pure simplicial complex of dimension $d-1$. 
Hence $S/I^{(\ell)}$ is an equidimensional ring of dimension $d$. 
So, it is enough to show that $(S/I^{(\ell)})_x$ is Cohen--Macaulay for 
any vertex $x$. 
Without loss of generality, we may assume that $x=x_{11}^{(1)}$. 
Then 
\begin{eqnarray*}
 I_{x}&=&(x_{1j}^{(1)}\,:\, 2 \le j \le n_{11})S_x
+(x_{ij}^{(k)} \,:\, 1 \le i \le d,\; 
1 \le j \le n_{ki},\; 2 \le k  \le p)S_x \\
& & + (x_{ij}^{(1)}x_{ij'}^{(1)} \,:\, 2 \le i \le d,\; 
1 \le j < j' \le n_{1i})S_x.  
\end{eqnarray*}
By Theorem \ref{SPcomplete} and Corollary \ref{Main1-extend}, $(S/I^{(\ell)})_x$ is Cohen--Macaulay 
for all $\ell \ge 1$.  
\end{proof}

\par \vspace{2mm}
\begin{exam} \label{special-exam}
Let $G=K_p$ be the complete $p$-graph. 
Then $\Delta_p$ is the complementary 
simplicial complex of $K_p$. 
Moreover, $\Delta_p$ has $p$ connected component$:$ 
$\Delta_p = \{x_1\} \coprod \ldots \coprod \{x_p\}$.  
Then $K[\Delta_p] = K[x_1,\ldots,x_p]/(x_ix_j\,:\, 1 \le i < j \le p)$. 
\par 
On the other hand, 
$K[\Delta_{\bf 1}] = K[x_1,\ldots,x_d]$, where ${\bf 1} = \underbrace{1,\ldots,1}_{d}$.  
\end{exam}

\par \vspace{1mm}
Now suppose that $S/I(G)^{(\ell)}$ has $($FLC$)$ for some $\ell \ge 3$. 
As $I(G) = \sqrt{I(G)^{(\ell)}}$, $S/I(G)$ also has $($FLC$)$ (see e.g. \cite{HTT}). 
Let $\Delta=\Delta(G)$ be the complementary simplicial complex of $G$: $I_{\Delta} = I(G)$. 
Then $\Delta$ is pure and 
$S_x/(I_{\Delta}^{(\ell)})_x$ is Cohen--Macaulay for every vertex $x \in V$. 
Put $\Gamma = \link_{\Delta} \{x\}$. 
This implies that $K[V \setminus \{x\}]/I_{\Gamma}^{(\ell)}$ is Cohen--Macaulay.  
Therefore, by Theorem \ref{Main-CM}, $I_x$ can be written as 
\[
 I_x = (y_1,\ldots,y_m)+ I(H_1)S_x + \cdots + I(H_{d-1})S_x,
\] 
where $H_1,\ldots,H_{d-1}$ are disjoint complete subgraphs of $G$ and  
$y_1,\ldots,y_m \in V$ such that $\{x,y_j\} \in E(G)$ 
and no elements of $\{y_1,\ldots,y_m\}$ are contained in $H_1  \cup \cdots \cup H_{d-1}$.     

\par \vspace{2mm}
In order to prove the second main theorem (Theorem \ref{Main-FLC}), we need the following lemma.

\begin{lemma} \label{Graph}
Let $G$ be a graph, and let $\Delta$ be the complementary simplicial complex of $G$: 
$I_{\Delta} = I(G)$. 
Suppose $d = \dim S/I(G) \ge 3$ and $\Delta$ is pure. 
Moreover, assume that for any vertex $u$, there exist vertices 
$y_1,\ldots,y_m$ and complete subgraphs $H_1,\ldots, H_{d-1}$ such that 
$I(G)_{u}$ can be written as 
\[
I(G)_{u} = (y_1,\ldots,y_m)S_u + I(H_1)S_u+\cdots + I(H_{d-1})S_u, 
\] 
where $V(G) = \{u\} \coprod \{y_1,\ldots,y_m\} \coprod V(H_1) \coprod \cdots \coprod V(H_{d-1})$.  
\par \vspace{2mm}
Then for any vertex $x \in V(G)$, there exist subgraphs $G_0$, $G_1$,\ldots,$G_d$ 
which satisfies the following conditions:
\begin{enumerate}
 \item $V(G) = V(G_0) \coprod V(G_1) \coprod \cdots \coprod V(G_{d-1}) \coprod V(G_d)$ and $x \in G_d$.  
 \item $G\,|_{V(G_i)} = G_i$ for each $i=0,1,\ldots,d-1,d$.  
 \item $G_1 \coprod \ldots \coprod G_{d-1} \coprod G_d$ is a disjoint union of complete graphs. 
 \item For every $y \in G_0$ and for every $z_i \in G_i$ $(i=1,\ldots,d)$, 
we have $\{y, z_i\} \in E(G)$. 
\end{enumerate}
\end{lemma}

\begin{proof}
Fix $x \in V(G)$. 
Applying the assumption to the case of $u=x$, we can find disjoint
complete subgraphs $G_1,\ldots,G_{d-1}$ of $G$
and vertices 
$y_1,\ldots,y_m$ such that 
\[
 I(G)_{x} = (y_1,\ldots,y_m)S_x+ I(G_1)S_x+\cdots + I(G_{d-1})S_x
\]
and $\{y_1,\ldots,y_m\}$ are contained in 
$V(G) \setminus V(G_1  \cup \cdots \cup G_{d-1})$.  
Then we prove the following claim. 

\begin{description}
\item[Claim 1] For any $y \in \{y_1,\ldots,y_m\}$, 
if $\{y,z_1\} \in E(G)$ for some $z_1 \in V(G_1)$, then 
$\{y, z_i\} \in E(G)$ holds for all $i=1,2,\ldots,d-1$ and 
for all $z_i \in V(G_i)$. 
\end{description}
\par 
Now suppose that $\{y,z_1\} \in E(G)$ for some 
$z_1 \in V(G_1)$. Then $yz_1 \in I(G)$.  
\par 
For any $z_i \in V(G_i)$ $(i=2,\ldots,d-1)$, 
if $\{y, z_i\} \notin E(G)$, 
then $yz_i \notin I(G)$. 
As $z_i \notin I(G)_x$, we have $xz_i \notin I(G)$. 
By the choice of $G_i$, $z_1z_i \notin I(G)$. 
Hence none of $y,x,z_1$ appears in $I(G)_{z_i}$. 
However, since $xy,yz_1 \in I(G)_{z_i}$, 
we have $xz_1 \in I(G)_{z_i}$ by assumption, and so $xz_1 \in I(G)$. 
This implies that $z_1 \in I(G)_x$. 
This contradicts the assumption. 
Thus we have $\{y,z_i\} \in E(G)$ for all $z_i \in V(G_i)$ $(i=2,\ldots,d-1)$. 
\par  
As $d \ge 3$, applying $\{y,z_2\} \in E(G)$ to the above argument, 
we obtain that $\{y,z'\} \in E(G)$ for all $z' \in V(G_1)$.  
Hence we proved the claim. 

\par \vspace{2mm}
By the above claim, by renumbering if necessary, 
we may assume that there exists an integer $k$ with $1 \le k \le m$ 
such that 
\begin{enumerate}
 \item[(i)] When $1 \le j \le k$,  
$\{y_j,z_i\} \in E(G)$ holds for every $1 \le i \le d-1$ and $z_i \in V(G_i)$. 
 \item[(ii)] When $k+1 \le j \le m$, 
$\{y_j,z\} \notin E(G)$ holds for every $z \in V(G_1 \cup \cdots \cup G_{d-1})$. 
\end{enumerate}

\par \vspace{2mm} \par \noindent 
Then we put $V_0=\{y_1,\ldots,y_k\}$ and $V_d = \{x,y_{k+1},\ldots,y_m\}$ and 
$G_0=G\,|_{V_0}$ and $G_d=G\,|_{V_d}$.  
In the following, we show that these $G_j$ $(j=0,\ldots,d)$ satisfy all conditions of the lemma. 
To show the condition (3), it is enough to show the following claim. 

\begin{description}
 \item[Claim 2] $G_d$ is a complete graph, and $G_i$ and $G_d$ are 
disjoint for each $i=1,\ldots,d-1$. 
\end{description}
To see that $G_d$ is a complete graph, 
it is enough to show that $\{u,u'\} \in E(G)$ whenever $u,u' \in V(G_d) \setminus \{x\}$. 
Suppose $\{u,u'\} \notin E(G)$. Take $z_1 \in V(G_1)$. 
Then since $\{x,z_1\}, \{u,z_1\}, \{u',z_1\} \notin E(G)$ and 
$xu,xu' \in I(G)_{z_1}$, we have $uu' \in I(G)_{z_1}$, and 
thus $\{u,u'\}\in E(G)$. 
The latter assertion immediately follows from the definition of $G_d$. 

To show the condition (4), it is enough to show the following claim. 
\begin{description}
 \item[Claim 3] For every $y \in G_0$, $\{y,u\} \in E(G)$ for every $u \in G_d$.  
\end{description}
Suppose that $\{y,u\} \notin E(G)$. 
Take $z_1 \in V(G_1)$ and $z_2 \in V(G_2)$. 
Then, since $d \ge 3$, $z_1,z_2,u$ are distinct vertices and $\{z_1,u\}$, $\{z_2,u\} \notin E(G)$ by Claim 2.  
By definition, $\{y,z_1\}$, $\{y,z_2\} \in E(G)$. 
By considering $yz_1$, $yz_2 \in I(G)_u$, we get $z_1z_2 \in I(G)_u$. 
Hence we have $\{z_1,z_2\} \in E(G)$. 
This is a contradiction. 
Therefore we conclude that $\{y,u\} \in E(G)$.

\par \vspace{2mm}
We have finished the proof of the lemma. 
\end{proof}

\par 
We are now ready to prove the second main theorem
in this paper.

\begin{thm} \label{Main-FLC} 
Let $G$ be a graph on $V=[n]$, 
and let $I(G)\subseteq S=K[v : v \in V]$ 
denote the edge ideal of $G$. 
Let $\Delta=\Delta(G)$ be the complementary simplicial complex of $G$, that is, 
$I_{\Delta} = I(G)$.   
Let $p$ denote the number of connected components of $\Delta$. 
Suppose that $\Delta$ is pure and $d=\dim S/I(G) \ge 3$. 
Then the following conditions are equivalent$:$ 
\begin{enumerate}
 \item $S/I(G)^{(\ell)}$ has $($FLC$)$ for every $\ell \ge 1$. 
 \item $S/I(G)^{(\ell)}$ has $($FLC$)$ for some $\ell \ge 3$. 
 \item There exist 
$(n_{i1},\ldots,n_{id}) \in \bbN^d$ for every $i=1,\ldots,p$
such that 
$\Delta$ can be written as 
\[
\Delta = \Delta_{n_{11},\ldots,n_{1d}} \;\textstyle{\coprod}\; 
\Delta_{n_{21},\ldots,n_{2d}} \;\textstyle{\coprod} 
\; \ldots \;
\textstyle{\coprod} \;
\Delta_{n_{p1},\ldots,n_{pd}}.
\] 
\end{enumerate}  
\end{thm} 

\begin{proof} 
$(3) \Longrightarrow (1)$: It follows from Proposition \ref{Main2-flc}. 
\par \vspace{1mm} \par \noindent 
$(1) \Longrightarrow (2)$ is clear. 
\par \vspace{1mm} \par \noindent 
$(2) \Longrightarrow (3)$: 
We may assume that $p \ge 2$ by Theorem \ref{Main-CM}. 
Then we note that $\Delta$ satisfies the assumption of Lemma \ref{Graph}. 
Fix $x \in V$. 
Let $G_0,\ldots,G_d$ be subgraphs of $G$ determined by Lemma \ref{Graph}. 
Then it suffices to show that the connected component containing $x$ (say, $\Delta'$) 
is the following form: $\Delta' = \Delta(G')$, 
where $G' = G_1 \cup \cdots \cup G_d$, which is a disjoint union of complete graphs.
\par 
First we see that $V(\Delta')=V(G_1 \cup \cdots \cup G_d)$. 
Let $z \in V(G_1 \cup \cdots \cup G_d)$. 
If $z \in V(G_1 \cup \cdots \cup G_{d-1})$, then as $\{x,z\} \notin E(G)$, $\{x,z\} \in \Delta'$, 
that is, $z \in V(\Delta')$.   
Otherwise, $z \in V(G_d)$. Then there exists a vertex $z' \in V(G_1 \cup \cdots \cup G_{d-1})$ 
such that $\{z,z'\} \notin E(G)$. 
Moreover, as $\{x,z'\} \in \Delta'$, we have $z \in V(\Delta')$. 
Hence $V(G_1 \cup \cdots \cup G_d) \subseteq V(\Delta')$. 
The converse follows from the condition (4) in Lemma \ref{Graph}. 
\par 
Next we see that $I_{\Delta'} = I(G_1 \cup \cdots \cup G_d)$. 
Since $\Delta'$ is a connected component of $\Delta$, we get 
\begin{eqnarray*}
 I_{\Delta'} &=& (I_{\Delta} \cap K[V(G_1 \cup \cdots \cup G_d)])S \\ 
&=& (I(G) \cap K[V(G_1 \cup \cdots \cup G_d)])S \\
&=& I(G_1 \cup \cdots \cup G_d). 
\end{eqnarray*}
This yields that $\Delta'= \Delta(G_1 \cup \cdots \cup G_d)$, as required. 
\end{proof}

\begin{remark}
Let $t$ be an indeterminate over $R$. 
If $R$ has $($FLC$)$ but not Cohen--Macaulay, then $R[t]$ does not have $($FLC$)$. 
Hence, in the above theorem, we cannot replace $S$ with $S[t]$, 
where $t$ is an indeterminate over $S$. 
\end{remark}

\par
Comparing Theorem \ref{Main-CM} and Theorem \ref{Main-FLC}, we obtain the following corollary. 

\begin{cor} 
Suppose that $d=\dim S/I(G) \ge 3$. 
Let $\Delta(G)$ denote the complementary simplicial complex of $G$. 
Let $\ell \ge 3$ be an integer. 
Then the following conditions are equivalent. 
\begin{enumerate}
 \item $S/I(G)^{(\ell)}$ has $($FLC$)$ and $\Delta(G)$ is connected. 
 \item $S/I(G)^{(\ell)}$ is Cohen--Macaulay.  
\end{enumerate}
Then $G$ is a disjoint union of finitely many complete graphs and 
$S/I(G)^{(k)}$ is Cohen--Macaulay for all $k \ge 1$. 
\end{cor}

\begin{remark}
In case of $\dim S/I(G)=2$, the following conditions are equivalent:
\begin{enumerate}
 \item $S/I(G)^{(\ell)}$ has $($FLC$)$ for every $\ell \ge 1$. 
 \item $S/I(G)^{(\ell)}$ has $($FLC$)$ for some $\ell \ge 1$. 
 \item $\Delta(G)$ is pure. 
\end{enumerate} 
In particular, we cannot remove the condition $d =\dim S/I(G) \ge 3$ 
from the assumption in Theorem \ref{Main-FLC}. 
For example, the pentagon cannot be expressed in the form as in Theorem \ref{Main-FLC}(3). 
\end{remark}

%

\subsection{FLC properties of ordinary powers}

\par 
In the rest of this section, we consider (FLC) properties of ordinary powers. 
Fix a positive integer $\ell$. Let $I=I_{\Delta}$ be a Stanley--Reisner ideal. 
If $S/I^{\ell}$ has (FLC), then 
$(S/I^{\ell})_x$ is Cohen--Macaulay for all vertex $x$. 
Then $I^{(\ell)}/I^{\ell}$ has finite length,  it is equal to 
$H_{\frm}^0(S/I^{\ell})$. 
Then $S/I^{(\ell)}$ also has (FLC). 
Hence we have the following theorem,
which gives an improvement of Goto--Takayama theorem in \cite{GT}
in the case of edge ideals.

\begin{thm} \label{ordinary-flc}
Put $d = \dim S/I(G) \ge 1$. 
Let $\Delta(G)$ denote the complementary simplicial complex of $G$. 
Then the following conditions are equivalent$:$
\begin{enumerate}
 \item $S/I(G)^{\ell}$ has $($FLC$)$ for every $\ell \ge 1$.  
 \item $S/I(G)^{\ell}$ has $($FLC$)$ for some $\ell \ge 3$. 
 \item $S/I(G)^{(\ell)}$ has $($FLC$)$ and $I(G)^{(\ell)}/I(G)^{\ell}$ has finite length 
for some $\ell \ge 3$. 
 \item $\Delta(G)$ is a pure, locally complete intersection complex. 
\end{enumerate}
\end{thm}

\begin{proof}
$(1) \Rightarrow (2) \Leftrightarrow (3)$ is clear. 
The equivalence of (1) and (4) follows from \cite{GT}. 
On the other hand, $(2) \Rightarrow (4)$ follows from 
Theorem \ref{Power-cor} by a similar argument as in \cite{GT}. 
\end{proof}

\begin{remark}
By Theorem \ref{Structure}, $(4)$ can be rephrased as follows:
\begin{enumerate}
\item[(4)']
When $d=2$, $\Delta(G)$ is a disjoint union of finitely many paths and $n$-gons with $n \ge 4$.  
When $d \ge 3$, $\Delta(G)$ is a disjoint union of finitely many complete intersection complexes of 
dimension $d-1$.  
\end{enumerate}
\end{remark}

\par 
The next example shows that there exists a graph $G$ for which 
$S/I(G)^{\ell}$ has (FLC) but \textit{not} Cohen--Macaulay. 

\begin{exam} \label{bipartite}
Under the notation as in Theorem $\ref{Main-FLC}$, 
$\Delta$ is locally complete intersection if and only if 
$\min\{n_{i1},\ldots,n_{id}\} \le 2$. 
\par 
For instance, for any positive integer $d$, 
the edge ideal of the 
complete bipartite graph $K_{d,d}$
\[
 I=(x_iy_j \,:\, 1 \le i,j \le d) \subseteq S=K[x_1,\ldots,x_d,y_1,\ldots,y_d] 
\]
satisfies the following statements$:$ 
\begin{enumerate}
 \item $S/I^{\ell}$ has (FLC) of dimension $d$ for every $\ell \ge 1$. 
 \item When $d \ge 2$, $S/I^{\ell}$ is \textit{not} Cohen--Macaulay for all $\ell \ge 1$.   
\end{enumerate}
\end{exam}

\begin{proof}
By \cite{SVV}, we know that $I^{(\ell)} = I^{\ell}$ for every $\ell \ge 1$; see also Lemma \ref{SVV}. 
Hence our theorem says that 
$S/I^{\ell}$ has $(FLC)$ for all $\ell \ge 1$. 
On the other hand, as $S/I$ is not Cohen--Macaulay,  
$S/I^{\ell}$ is \textit{not} Cohen--Macaulay if $d \ge 2$ and $\ell \ge 1$. 
\end{proof}

\par 
Even if $S/I(G)$ is Cohen--Macaulay, one can find an example of $G$ such that 
$S/I(G)^{\ell}$ has (FLC) but not Cohen--Macaulay.  

\begin{exam} \label{4-pointed}
Let $\Delta$ be a $4$-pointed path, and $I_{\Delta}=(x_1x_2,x_2x_3,x_3x_4)$. 
Then $I_{\Delta}$ is also the edge ideal of the $4$-pointed path $G$. 
Then $S/I_{\Delta}$ is Cohen--Macaulay, and $S/I_{\Delta}^{2}$ is 
Buchsbaum (thus (FLC)) but \textit{not} 
Cohen--Macaulay.     
\par 
Similarly, for the pentagon $G$, $S/I(G)$ is Cohen--Macaulay and $S/I(G)^{3}$ 
has (FLC) but not Cohen--Macaulay.  
\end{exam}

\par \vspace{2mm}
In general, even if $S/I(G)^{(\ell)}$ has (FLC), it is not necessarily $S/I(G)^{\ell}$ has (FLC)
as the next example shows. 
Note that we can construct similar examples of graphs $G$ with $\dim S/I(G)=d$
for every $d \ge 3$.  

\begin{exam} \label{doubledelta}
Let $S=K[\{x_i\}_{1 \le i \le 9},\{y_j\}_{1 \le j \le 9}]$, and
let $G$ be a graph such that $\Delta(G) = \Delta_{3,3,3} \coprod \Delta_{3,3,3}$. 
Set   
\begin{eqnarray*}
I(G) &=& (x_1x_2,x_1x_3,x_2x_3,x_4x_5,x_4x_6,x_5x_6,x_7x_8,x_7x_9,x_8x_9) \\ 
   && + (y_1y_2,y_1y_3,y_2y_3,y_4y_5,y_4y_6,y_5y_6,y_7y_8,y_7y_9,y_8y_9) \\
   && + (x_iy_j \,:\, 1 \le i,j \le 9). 
\end{eqnarray*}
Then 
\begin{enumerate}
\item $\dim S/I(G)=3$. 
\item $S/I(G)^{(\ell)}$ has (FLC) for every $\ell \ge 1$. 
\item $\Delta(G)$ is not a locally complete intersection complex.  
\item $S/I(G)^{\ell}$ does not have (FLC) for every $\ell \ge 3$. 
\end{enumerate}
\end{exam}

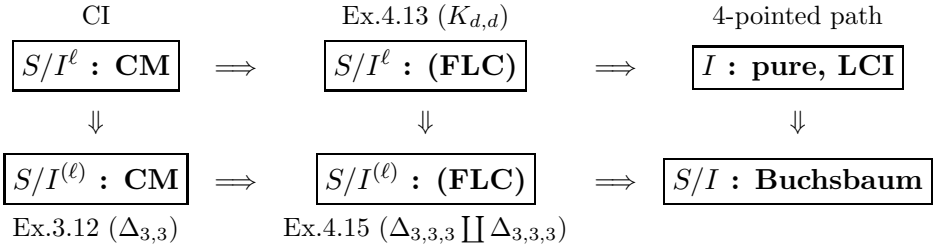
\begin{flushleft}
\begin{tabular}{ccccc}
CI & & Ex.\ref{bipartite} ($K_{d,d}$) & & $4$-pointed path \\[1mm]
\framebox{\bf \large $S/I^{\ell}$ : CM} & $\Longrightarrow$ & 
\framebox{\bf \large $S/I^{\ell}$ : (FLC)} &  $\Longrightarrow$ & 
\framebox{\bf \large $I$ : pure, LCI} \\[3mm]
$\Downarrow$ & & $\Downarrow$  & & $\Downarrow$  \\[2mm]
\framebox{\bf \large $S/I^{(\ell)}$ : CM} & $\Longrightarrow$ & 
\framebox{\bf \large $S/I^{(\ell)}$ : (FLC)} & $\Longrightarrow$  
& \framebox{\bf \large $S/I$ : Buchsbaum} \\[2mm]
Ex.\ref{union-triangles} ($\Delta_{3,3}$) & & Ex.\ref{doubledelta} ($\Delta_{3,3,3} \coprod \Delta_{3,3,3}$) & & 
\end{tabular}
\end{flushleft}


\par \vspace{2mm}
\begin{acknowledgement}
The second author was supported by JSPS 20540047. 
The third author was supported by JSPS 19340005. 
The authors would like to express their gratitude to the referee for
his careful reading. 
\end{acknowledgement}


\end{document}